\documentclass[12pt]{amsart}

\usepackage{amsmath}
\usepackage{amssymb}
\usepackage{amscd}
\usepackage{graphicx}
\usepackage{calrsfs}
\usepackage{enumitem}
\usepackage{xcolor}

\newtheorem*{TA}{\sc theorem A}
\newtheorem*{TB}{\sc theorem B}
\newtheorem*{TC}{\sc theorem C}

\newtheorem{thm}{\sc theorem}[section]
\newtheorem{pro}[thm]{\sc proposition}
\newtheorem{lem}[thm]{\sc lemma}
\newtheorem{cor}[thm]{\sc corollary}

\theoremstyle{definition}

\newtheorem{dfn}[thm]{\sc definition}
\newtheorem{ntt}[thm]{\sc notation}

\newtheorem{qst}[thm]{\sc question}

\theoremstyle{remark}

\newtheorem{rmk}[thm]{\sc remark}

\newcommand{\ind}{\mbox{$\mbox{\rm ind}$}}

\def\im{{\rm Im}}

\def\crit{{\rm Crit}}

\def\nb{{\mathcal Op}}

\def\S{{\bf S}}
\def\D{{\bf D}}
\def\I{{\bf I}}

\def\CC{{\mathcal C}}

\def\FF{{\mathcal F}}
\def\F{{\mathcal F}}

\def\HH{{\mathcal H}}

\def\LL{{\mathcal L}}

\def\OO{{\mathcal O}}

\def\ZZ{{\mathcal Z}}

\def\id{{\rm id}}
\def\mun{{^{-1}}}

\newcommand{\EL}{{\mathcal L}{}}

\newcommand{\R}{\mathbb R}

\newcommand{\W}{{\mathcal W}{}}
\newcommand{\Z}{\mathbb Z}

\newcommand{\rp}{respectively }
\renewcommand{\im}{{\rm Im}}
\newcommand{\om}{\omega}

\begin{document}

\title[Conformal Symplectic, Foliated and Contact Structures]
{Conformal Symplectic structures, Foliations and Contact Structures}

\author{M\'elanie Bertelson, Ga\"el Meigniez}

\date{\today}


\begin{abstract} This paper presents two existence h-principles, the first for conformal symplectic structures on closed manifolds, and the second for leafwise conformal symplectic structures on foliated manifolds with non empty boundary. The latter h-principle allows to linearly deform certain codimension-$1$ foliations to contact structures. These results are essentially applications of the Borman-Eliashberg-Murphy h-principle for overtwisted contact structures and of the Eliashberg-Murphy symplectization of cobordisms, together with tools pertaining to foliated Morse theory, which are elaborated here.
\end{abstract}
\maketitle

\section{introduction}

In high dimensions, symplectic and contact topology require modern methods different from those effective in dimension three. In the present paper,
we essentially explore some consequences of the Eliashberg-Murphy symplectization of cobordisms \cite{eliashberg_murphy_20} together with
 the Borman-Eliashberg-Murphy h-prin\-ciple for overtwisted contact structures \cite{borman_eliashberg_murphy_15}.
  Also, one needs some tools falling to the
Morse theory of cod\-im\-ension-one foliations, which we elaborate.

\subsection{Existence of conformal symplectic structures} 

 On a manifold $M$,
a \emph{conformal symplectic structure} is
 a conformal class of nondegenerate $2$-forms,
 admitting a local symplectic representant in a neighborhood of every point.
This generalization of symplectic geometry is classical
 (see for example \cite{vaisman_85}, \cite{banyaga_02} and \cite{chantraine_murphy_19}). To such a structure,
one associates its \emph{Lee class:} a real cohomology
class of degree $1$ on $M$, which is
the obstruction to finding a global symplectic representant.

\medskip
Theorem $A$ in Section \ref{conformal_sec} is an existence h-principle for conformal symplectic structures on closed manifolds whose Lee class is any \underline{nonzero} de Rham cohomology class of degree one. The only case excluded is the one that would yield a genuine symplectic structure. We thus generalize a result obtained by Eliashberg-Murphy for Lee classes of rank one over $\Z$ (Theorem~1.8 in \cite{eliashberg_murphy_20}).

\subsection{Making foliated cobordisms conformally symplectic}

The data are a cobordism $(W, \partial_-W, \partial_+W)$ endowed with a codimension-$1$ coorientable taut foliation $\F$, whose leaves all meet both $\partial_-W$ and $\partial_+W$ transversely, and with a leafwise closed $1$-form $\eta$. Theorem $B$ in Section \ref{cobordism_sec} is an existence h-principle for leafwise conformal symplectic structures whose Lee class in every leaf is the cohomology class of $\eta$;
the boundary component
 $\partial_-W$ (\rp $\partial_+W$) being leafwise convex (resp. concave) in a certain sense.
 In particular, the leaves of $\F\vert_{\partial_-W}$  (\rp $\F\vert_{\partial_+W}$) are
 positive (resp. negative) contact submanifolds of $W$.

\medskip
Two tools belonging to the Morse theory of codimension-$1$ foliations are essential here. A real function $f$ on $W$ is called \emph{leafwise Morse} if its restriction to every leaf of $\F$ is a Morse function in the leaf. The tools in question are the construction of \emph{ordered} leafwise Morse functions, and a cancellation method for leafwise local extrema, analogous to the Cerf cancellation of local extrema in $1$-parametric families of functions. The first tool also reproves the existence of faithful submanifolds in taut codimension-$1$
foliations \cite{meigniez_16}.
 Also, these tools show that for any taut codimension-$1$ foliation on a compact manifold with boundary, if all the leaves meet the boundary transversely and non-trivially, then this foliation is uniformly open in the sense of \cite{bertelson_02};
hence, the h-principle for open leafwise invariant differential relations is verified for such foliations.

\subsection{Deforming foliations to contact structures}

Given a cod\-im\-ens\-ion-$1$ foliation $\F$ on a manifold $M$, defined by a non-vanishing $1$-form $\alpha$, a \emph{linear contactizing deformation} is a $1$-parameter family $\alpha_t = \alpha + t\lambda$ of $1$-forms such that $\alpha_t$ is contact for all sufficiently small positive $t$. These deformations are completely understood in dimension $3$
(see \cite{eliashberg_thurston_98}), but not much is known in higher dimensions.
 
\medskip
A linear contactizing deformation is provided by any \emph{exact}
 leafwise conformal symplectic structure whose Lee class is precisely the linear holonomy
  of the foliation. Here, ``exact'' refers to the Lichnerowicz differential in the leaves.
  Indeed, the deformation $\alpha + t \lambda$ is contact for all sufficiently small positive $t$'s if $d\lambda-\eta\wedge\lambda$ is leafwise nondegenerate. Moreover, such leafwise conformal structure exist (Theorem $C$ in Section \ref{deformation_sec}) for a large class of  foliations on cobordisms, which we call \emph{holonomous}. Although the leafwise \emph{convex} boundary may be empty, the leafwise
concave one cannot. Unfortunately, it seems that
 our construction is not able to produce contactizing deformations on closed manifolds.\\

{\bf Aknowledgments} This work was carried out partly while the second author was the host of the Mathematics Research Center in Stanford, and partly while both authors were hosts of the \emph{Forschunginstitut f\"ur Mathematik} in ETH Z\"urich. We thank both institutions. The first author is supported by the \emph{Excellence of Sciences} project \emph{Symplectic Techniques in Differential Geometry} jointly funded by the F.R.S.-FNRS and the FWO. It is a pleasure to thank Yasha Eliashberg for his observations and support. \\
\begin{ntt}
One denotes herefater by $\I$ the compact interval $[0,1]$, by $\D^d\subset\R^d$ the unit compact ball, and by $\S^{d-1}$ the sphere $\partial\D^d$. Every foliation $\F$ on a manifold $M$ with boundary is assumed to be transverse to the boundary, inducing therefore a foliation $\partial\F = \F\vert_{\partial M}$ on $\partial M$. The notation $\nb_X(Y)$, or $\nb(Y)$
if $X$ is understood,
holds for ``some open neighborhood of $Y$ in $X$''.
\end{ntt}

\section{Elements of conformal symplectic geometry}\label{elements_sec}

Let $\eta$ be a closed $1$-form on a manifold $M$ of dimension $2n\ge 2$. The \emph{Lichnerowicz} (also known as \emph{Novikov}) \emph{differential with respect to $\eta$}, denoted by $d_\eta$, is the ordinary
 Cartan differential somehow twisted by $\eta$; precisely,  
for every differential form $\theta\in\Omega^*(M)$:
$$
d_\eta\theta=d\theta-\eta\wedge\theta
$$
It is immediately verified that $d_{\eta}^2 \equiv 0$. The cohomology of the differential operator $d_\eta$ is called the \emph{Novikov cohomology of $M$ with respect to $\eta$} and denoted 
by $H^*_{\eta}(M)$.  

\begin{rmk}\label{cc_rmk}
The Novikov cohomology
 only depends, up to isomorphism, on the de Rham cohomology class of $\eta$.
 Precisely, when $\eta'-\eta=dF$ for some smooth
 function $F$ on $M$,  the differential complexes $(\Omega^*(M), d_{\eta'})$ and $(\Omega^*(M), d_\eta)$
  are isomorphic through a conformal rescaling of the differential forms~:
$$
d_{\eta'}\theta=e^F d_{\eta}(e^{-F}\theta).
$$
(Beware that the resulting
 isomorphism between $H^*_{\eta}(M)$ and $H^*_{\eta'}(M)$ is not canonical,
  depending on the choice of $F$.)
\end{rmk}

We have not found in the litterature the following generalization, which we shall need
further down, of the Poincaré Lemma.

\begin{lem}\label{poincare_lem} Let $M'$ be a manifold, $h:M'\times\I\to M$
 be
a smooth map. Define
$$F:M'\times\I\to\R:(x,t)\mapsto\int_0^t\eta(\frac{\partial h}{\partial t}(x,\tau))d\tau$$
$$\HH:\Omega^*(M)\to\Omega^{*-1}(M'):\theta\mapsto
\int_0^1e^{-F}\iota_{\partial/\partial t}(h^*(\theta)) dt$$
Consider on $M'$ the function $F_1:x\mapsto F(x,1)$, and for every $t\in\I$ the map
$h_t:x\mapsto h(x,t)$, and the $1$-form
 $\eta'_0=h_0^*(\eta)$.
 
  Then:
 
i) $e^{-F_1}h_1^*(\theta)-h_0^*(\theta)=d_{\eta'_0}\HH(\theta)+\HH(d_\eta\theta)$;

ii) The morphisms $\theta\mapsto h_0^*(\theta)$ and $\theta\mapsto e^{-F_1}h_1^*(\theta)$
induce the same morphism in Novikov cohomology $$H^*_\eta(M)\to H^*_{\eta'_0}(M')$$
\end{lem}

\begin{proof}[Proof of Lemma \ref{poincare_lem}] Let $\theta\in\Omega^*(M)$.
 On $M'\times\I$, consider the forms $\eta'=h^*(\eta)$ and
  $\theta'=e^{-F}h^*(\theta)$. Writing for short $\iota_t$ instead of
   $\iota_{\partial/\partial t}$,
  develope the Cartan formula:
$$\LL_{\partial/\partial t}(\theta')=d\iota_{t}(\theta')+\iota_{t}(d\theta')$$
$$
=e^{-F}\big(
d\iota_t(h^*\theta)-dF\wedge\iota_t(h^*\theta)+\iota_t(d(h^*\theta))-\iota_t(dF\wedge
h^*\theta)
\big)
$$
Since moreover
$$\iota_t(dF)=\frac{\partial F}{\partial t}=\eta(\frac{\partial h}{\partial t})=\iota_t(\eta')$$
defining $
H(\theta)=e^{-F}\iota_t(h^*\theta)
$, one gets
$$\LL_{\partial/\partial t}(\theta')=d_{\eta'}H(\theta)+H(d_\eta\theta)$$

 Of course, (i) follows by restriction to the slices $M'\times t$ and integration
 with respect to $t$; then, (ii) follows immediately from (i).
\end{proof}

If $M$ has a boundary, the \emph{relative} Novikov cohomology
 $H_\eta^*(M,\partial M)$ with respect to $\eta$ is defined, as usual, as the cohomology of 
  $$\Omega^*(M) \times \Omega^{*-1}(\partial M)$$ under the differential operator
\begin{equation}\label{relative_eqn}
D_\eta:(\theta,\theta')\mapsto(d_\eta\theta,\theta\vert_{\partial M}-d_\eta\theta')
\end{equation}

A number of standard notions in symplectic geometry admit obvious generalizations to conformal symplectic geometry: one simply replaces the Cartan operator $d$ by its twisted version $d_\eta$.

\begin{dfn}\label{twisted_dfn} A $2$-form $\omega$ on $M$ is \emph{$\eta$-symplectic} if $\omega$ is $d_\eta$-closed and nondegenerate.
\end{dfn}

Every $\eta$-symplectic form $\omega$
  defines a conformal symplectic structure on $M$.
  Indeed, after Remark \ref{cc_rmk}, for every open subset $U$ of $M$
 on which $\eta$ admits a primitive $F$,
  the $2$-form $e^{-F}\om$ is closed, and
   hence genuinely symplectic on $U$. 
   
   Conversely, for every nondegenerate $2$-form $\om$
   representing a conformal symplectic structure,
   one has a closed $1$-form $\eta$ on $M$
   whose local primitives $F$ satisfy $d(e^{-F}\om)=0$.
    In other words, $\om$ is $d_\eta$-closed.
   If $n\ge 2$, then $\eta$ is uniquely defined by $\omega$,
    and called \emph{the Lee form of $\omega$}.

\begin{dfn}\label{twisted_dfn_2}
A $1$-form $\lambda$ on $M$ is \emph{$\eta$-Liouville} if $d_\eta\lambda$ is nondegenerate. Its \emph{$\eta$-Liouville vector field,} or \emph{$\eta$-dual,} $Z = Z_\lambda$, is defined by the relation
$$\lambda=\iota_Z(d_\eta\lambda).$$

If moreover $M$ is oriented, $\lambda$
 is of course
  called \emph{positive} if the volume form $(d_\eta\lambda)^n$ defines the given orientation;
 and \emph{negative} otherwise.
\end{dfn}
Notice that, in any open subset of $M$ where $\eta$ admits a primitive $F$, the 
$\eta$-Liouville vector field $Z$ is nothing but
 the ordinary Liouville vector field of the ordinary Liouville form
 $e^{-F}\lambda$. In other words, the dual Liouville vector field is invariant
  by conformal equivalence of conformally Liouville forms. Besides,
 after Cartan's formula~:
\begin{equation}\label{lie_eqn}
\begin{array}{lll}
\LL_Z\lambda = \iota_Z(d\lambda) = \iota_Z(d_\eta\lambda+\eta\wedge\lambda) = (1+\eta(Z))\lambda \\
\LL_Z(d_\eta\lambda) = d\LL_Z\lambda-(\LL_Z\eta)\wedge\lambda-\eta\wedge\LL_Z\lambda = (1+\eta (Z))d_\eta\lambda.
\end{array}\end{equation}
Actually, the second equation
 is equivalent to $Z$ being a $\eta$-Liouville vector field for $\om$;
  and the first one too, under the hypothesis $\lambda(Z) = 0$.

\begin{rmk}\label{dR}
After Remark \ref{cc_rmk},
the existence of a $\eta$-Liouville $1$-form
 (\rp $\eta$-sympl\-ectic $2$-form) does not
  depend on the choice of the form $\eta$ in its de Rham cohomology class.
 Precisely, if $\lambda$ (\rp $\om$) is a
  $\eta$-Liouville $1$-form (\rp $\eta$-symplectic $2$-form), then for any function $F$ on $M$,
  the form $e^F \lambda$ (\rp $e^F\om$) is $(\eta + dF)$-Liouville (\rp $(\eta + dF)$-symplectic).
\end{rmk}

\begin{rmk}
One can alternatively interpret a $\eta$-symplectic (\rp $\eta$-Liouville) form on $M$ as a genuinely  symplectic (\rp Liouville) and \emph{equivariant}
 form on the abelian covering of $M$ defined by $\eta$, or on the universal cover.
  This alternative viewpoint does not seem to be the most efficient for the present paper.
\end{rmk}
 
\begin{dfn}\label{twisted_dfn_3} Let $\omega$ be a  $\eta$-symplectic form
on $M$ and $H\subset M$ be a cooriented
 hypersurface. As usual,
  one orients $H$ by the volume form $\iota_X(\omega^n)$,
   where $X$ is transverse to $H$ and positive (with respect to the coorientation).
    One calls $H$ of \emph{convex} (\rp \emph{concave}) \emph{contact type}
  with respect to  $\omega$ if $H$ is transverse to a
  positive (resp. negative) $\eta$-Liouville vector field $Z$ defined near $H$.
  \end{dfn}

\medbreak
\emph{Even contact structures} will play a crucial role in our construction. Recall that they are defined as maximally non-integrable hyperplane fields on even-di\-men\-sio\-nal manifolds. We will assume that the ambient manifold $M$ is $2n$-dimensional ($n\ge 1$)
 and oriented, and that the even contact structure $\varepsilon$
  is cooriented, so that $\varepsilon = \ker \lambda$
   for some non-vanishing $1$-form $\lambda$ defining the coorientation
   (\emph{even contact form}).
    The maximal non-integrability of $\varepsilon$ is equivalent to the non-vanishing of $\lambda \wedge d\lambda^{n-1}$. The
   dimension-$1$-foliation $\ZZ$ spanned by the rank-$1$ distribution
    $T\ZZ=\ker(d\lambda\vert_\varepsilon)$ is
     called the \emph{characteristic foliation} of $\varepsilon$. The foliation $\ZZ$ is transversely contact: $\varepsilon$
is invariant by any vector field tangential to $\ZZ$, and $\lambda\wedge d\lambda^{n-1}$ defines a volume
form on $TM/T\ZZ$.
We fix an orientation for $\ZZ$, namely,
 a section $Z$ of $T\ZZ$ is positive if $Z$ followed by a basis of $TM/\ZZ$ which is
 positive with respect to $\lambda \wedge d\lambda^{n-1}$, makes a positive basis of $M$.

\medskip
Let $R$ be any vector field on $M$ such that $\lambda(R) \equiv 1$,
 let $Z$ be any positive section of $T\ZZ$; choose any $1$-form $\theta$ on $M$
  such that the function $\theta(Z) + d\lambda (Z, R)$ is positive on $M$.
   Equivalently, the $1$-form $\theta - \iota_R(d\lambda)$ is positive on $\ZZ$.
  Then, the $2$-form 
\begin{equation}\label{omega_eqn}
\om_\theta = \theta \wedge \lambda + d\lambda
\end{equation}
is positive nondegenerate. Indeed,
$$\om_\theta^n = (\theta \wedge \lambda + d\lambda)^n = n \theta \wedge \lambda \wedge d\lambda^{n-1} + d\lambda^n$$ 
is positive, since $$\iota_R(\iota_Z(\om_\theta^n)) = n (\theta (Z) + d\lambda(Z, R)) d\lambda^{n-1}$$ induces a positive form on $\varepsilon /T\ZZ$, as a consequence of our choice of orientation for $\ZZ$. The homotopy class 
of  nondegenerate $2$-forms represented by $\om_\theta$
 depends only on $\varepsilon$ and of its coorientation, and not on the choices of the
  forms $\lambda$ nor $\theta$.
   Call this class the \emph{almost symplectic class associated to $\varepsilon$}. 
   
    Finally, fix a positively oriented nonsingular section $Z$ of $T\ZZ$. Since $\varepsilon$ is $Z$-invariant, for every defining form $\lambda$,
    one has a unique function $\chi_\lambda$ on $M$ such that
$$\LL_Z\lambda=\chi_\lambda\lambda.$$
If one likes better,
\begin{equation}\label{chi_eqn}
\chi_\lambda = d\lambda(Z, R) = - \lambda([Z, R])
\end{equation}
where $R$
is any vector field on $M$ such that $\lambda(R)=1$.
Clearly, for every function $u$ on $M$:
\begin{equation}\label{derivate_eqn}
\chi_{e^u\lambda} = \chi_\lambda + Z (u).
\end{equation}
\medbreak
 
 The method used in the present paper to build an $\eta$-Liouville form consists of
  two steps,
 whose details will be given in Lemmas \ref{even_contact_lem} and \ref{derivate_lem}.
 We use an auxiliary ambiant codimension-$1$ foliation $\F$.
 Essentially, first, the h-principle for overtwisted contact structures
 yields an even contact structure $\varepsilon$ whose characteristic foliation is
 transverse to $\FF$. Second, a $\eta$-Liouville form
 is found among the $1$-forms representing $\varepsilon$.

\begin{rmk} McDuff's early
 h-principle for even contact structures
  (\cite{McD87}, Proposition 7.2) does unfortunately not seem to
   allow such control on
 the characteristic foliation.
\end{rmk}

\begin{dfn}\label{ac_dfn}
Given a codimension-$1$ foliation $\F$ on a manifold $M$,
recall the \emph{leafwise} (also known as ``foliated'', or ``tangential'')
 differential forms: $\Omega^k(\F)$ stands for the collection of
the smooth sections of $\Lambda^k(T\F)$. For $\theta\in\Omega^*(M)$,
one has the restriction $\theta\vert_\F\in\Omega^*(\F)$. For 
$\theta\in\Omega^*(\F)$, one has the \emph{leafwise Cartan differential}
 $d_\F\theta\in\Omega^{*+1}(\F)$,
 such that $d_\F(\theta\vert_\F)=(d\theta)\vert_\F$. The differential operator $d_\F$
 on $\Omega^*(\F)$
  defines the leafwise
 (also known as ``foliated'') cohomology $$H^*(\F)=\ker(d_\F)/\im(d_\F)$$
 
 A leafwise $1$-form $\alpha\in\Omega^1(\F)$ is \emph{contact}
 if the leafwise ($2n-1$)-form $\alpha\wedge(d_\F\alpha)^{n-1}$ does not vanish,
 where $\dim(M)=2n$.
 
 An \emph{almost contact structure} on a manifold
 $\Sigma$ of dimension  $2n-1$ is a pair $(\alpha,\varpi)\in\Omega^1(\Sigma)\times\Omega^2(\Sigma)$
such that $\alpha\wedge\varpi^{n-1}$ does not vanish.

 In the same way, a \emph{leafwise almost contact structure} on 
 $\F$ is a pair $(\alpha,\varpi)\in\Omega^1(\F)\times\Omega^2(\F)$
such that the leafwise ($2n-1$)-form $\alpha\wedge\varpi^{n-1}$ does not vanish.

\end{dfn}

\begin{rmk}\label{lwctctandnd} In a real vector space $E$ of dimension $2n$,
given a cod\-im\-ens\-ion-$1$ vector subspace $\Sigma\subset E$
and a vector $Z\in E$ not in $\Sigma$, let
$\theta$ be the linear form of kernel $\Sigma$ such that $\theta(Z)=1$,
and let $\pi:E\to\Sigma$ be the linear projection parallel to $Z$.
Then, the linear mappings
$$\Lambda^2(E)\to\Lambda^1(\Sigma)
\times\Lambda^2(\Sigma):
\omega\mapsto(\iota_Z(\omega)\vert_\Sigma,\omega\vert_\Sigma)$$
$$\Lambda^1(\Sigma)
\times\Lambda^2(\Sigma)\to\Lambda^2(E):
(\alpha,\varpi)\mapsto\theta\wedge\pi^*(\alpha)+
\pi^*(\varpi)$$
are reciprocal isomorphisms. Moreover, $\omega^n\neq 0$ if and only if
the corresponding pair $(\alpha,\varpi)$ satisfies $\alpha\wedge\varpi^{n-1}\neq 0$.

Consequently, given a cooriented codimension-$1$ foliation $\F$ on the $(2n)$-manifold $M$,
choose a nonvanishing $1$-form $\theta$ on $M$ defining $\F$
and a vector field $Z$ such that $\theta(Z)=1$; and denote by $\pi$  the projection $TM\to T\F$
parallel to $Z$.
Then, the linear mappings
$$\Omega^2(M)\to\Omega^1(\F)
\times\Omega^2(\F):
\omega\mapsto(\iota_Z(\omega)\vert_\F,\omega\vert_\F)$$
$$\Omega^1(\F)
\times\Omega^2(\F)\to\Omega^2(M):
(\alpha,\varpi)\mapsto\theta\wedge\pi^*(\alpha)+
\pi^*(\varpi)$$
are reciprocal isomorphisms. Moreover, $\omega$ is an almost symplectic structure on $M$
 if and only if
 the corresponding pair $(\alpha,\varpi)$ is a leafwise almost contact structure on $\F$.
\end{rmk}
 
 \begin{lem}\label{even_contact_lem}
  (i) On a manifold $M$, let $\omega$ be a nondegenerate
 $2$-form, and $\F$ be a cooriented codimension-$1$
 foliation.
 
 Then, $M$ admits  an even contact structure $\varepsilon$ such that
 \begin{itemize}
 \item $\omega$ lies in the almost symplectic class associated with
  $\varepsilon$;
  \item
The characteristic foliation of $\varepsilon$ is positively transverse to $\FF$;
 \item On every leaf $L$ of $\F$,  the contact structure $\varepsilon\cap L$ is overtwisted.
 \end{itemize}
 
 (ii) Moreover, given a closed subset $A\subset M$,
  assume that an even contact structure $\varepsilon_A$
  is already given on some open
 neighborhood $U$ of $A$,
 such that
 
  \begin{itemize}
 \item  $\omega\vert_U$ lies in the almost symplectic class associated with $\varepsilon_A$;
  \item
The characteristic foliation of $\varepsilon_A$ is positively transverse to $\FF\vert_U$;
 \item On every leaf $L$ of $\F$ which is entirely contained in $U$, if any,
   the contact structure $\varepsilon_A\cap L$ is overtwisted.
 \end{itemize}
   
  Then, one can choose $\varepsilon$ to coincide with
 $\varepsilon_A$ on some smaller neighborhood of $A$.
 \end{lem}
 
 Note --- For the case $\dim(M)=2$: one agrees that any nonsingular $1$-form
on a $1$-manifold is an \emph{overtwisted} contact form on this manifold.

 \begin{proof}[Proof of Lemma \ref{derivate_lem}]
 
(i) Choose a vector field $Z$ over $M$, positively transverse to $\F$. The
ambiant nondegenerate $2$-form $\omega$
  induces on $\F$ a leafwise almost contact structure
  $(\iota_Z(\omega)\vert_\F,\omega\vert_\F)$ (Remark \ref{lwctctandnd}).
 After the h-principle for overtwisted contact structures
  on foliations
(\cite{borman_eliashberg_murphy_15}, Theorem 1.5), $M$ carries
a leafwise contact structure $\xi\subset{T}\F$
 (a cooriented $(2n-2)$-plane field defining on every leaf
 a contact structure) which is overtwisted in every leaf, and homotopic to
  $(\iota_Z(\omega)\vert_\F,\omega\vert_\F)$ as a leafwise almost contact structure. 
 Let  $\varepsilon$ be on $M$ the hyperplane field spanned by $\xi$ and $Z$,
 and $\lambda$ be a nonvanishing $1$-form defining $\varepsilon$ and its coorientation.

We claim that there is a unique vector field $X$ over $M$, contained in $\xi$,
and such that the flow $(\phi^t)$ of $Z'=Z+X$ preserves $\varepsilon$.

Indeed, this foliated version of the Gray
stability theorem can be proved, like the classical one, by a Moser-type argument:

The condition $(\phi_t)^*(\varepsilon)=\varepsilon$ amounts to $\varphi_t^*\lambda = g_t\lambda$, for some smooth family of positive functions $(g_t)$.
 Derivating with respect to $t$ yields~:
$$\phi_t^*(\EL_{Z'} \lambda) = \dot{g_t} \lambda,$$
which implies that
\begin{equation}\label{Gt}
\iota_{Z'}(d\lambda) = G_t \lambda
\end{equation}
where
$G_t = ({\dot{g_t}}/{g_t})\circ\phi_t^{-1}$. 
Let $R$ be the vector field on $M$ tangential to $\F$, and which is, in
every leaf $L$ of $\F$, the Reeb vector field of the contact form $\lambda\vert_L$.
Since ${Z'} = Z + X$, we can rewrite (\ref{Gt}) as a system of two equations~:
$$\left\{\begin{array}{lll}
\iota_X(d \lambda)|_\xi = 
- \iota_Z(d\lambda)|_\xi \\
d \lambda(Z,R) = G
\end{array}\right.$$
The first one uniquely determines $X$ since $\lambda$ is leafwise contact;
 the second one determines $G$.  The
 claim is proved.
 
 Clearly, $\varepsilon$ is an even contact structure whose characteristic foliation
 is positively spanned by $Z'$. The almost symplectic class
 associated with $\varepsilon$ does coincides with that of $\om$, thanks to
 Equation (\ref{omega_eqn}) and
 Remark \ref{lwctctandnd}.

(ii) The plane field $\xi_A=\varepsilon_A\cap T\F$
is contact in every leaf of $\F\vert_U$.
 After the unicity up to isotopy of overtwisted contact structures on foliations
(\cite{borman_eliashberg_murphy_15}, Theorem 1.6),  after pushing $\xi$
by an isotopy of $M$ tangential to $\F$,
we can arrange that $\xi=\xi_A$ on $\nb_M(A)$. Then,
 choose $Z$ so that
it spans positively the characteristic foliation of $\varepsilon_A$
 on $\nb_M(A)$. Hence, $\varepsilon=\varepsilon_A$ on $\nb_M(A)$.
 After the local unicity property of $X$, one has $X=0$ on $\nb_M(A)$.
 Hence,  $Z'=Z$ on $\nb_M(A)$.
\end{proof}

\begin{lem}\label{derivate_lem} On the even-dimensional oriented
manifold $M$, let $\eta$ be a closed $1$-form,
and let $\lambda$ be a non-vanishing $1$-form.

i) Assume that $\lambda$ is $\eta$-Liouville. Then, $\lambda$
 yields an even contact structure $\ker \lambda$, whose characteristic foliation
 is positively spanned by the
 $\eta$-Liouville vector field $Z_\lambda$, and
  whose
 associated almost symplectic class is represented by $d_\eta \lambda$.

ii) Conversely, assume that $\lambda$ defines an even contact structure. Fix a positive section $Z$ of the characteristic foliation. Then, $\lambda$ is $\eta$-Liouville positive if and only if $\chi_\lambda > \eta (Z)$. Moreover, if so, the $\eta$-Liouville vector field $Z_\lambda$
 $\eta$-dual
to $\lambda$ coincides with $(\chi_\lambda-\eta (Z))\mun Z$.

\end{lem}
    
\begin{proof} Let  $R$ be a vector field on $M$ such that $\lambda(R) \equiv 1$.

 (i) Since $\lambda$ is $\eta$-Liouville,
 the $(2n-1)$-form $$n \lambda \wedge (d\lambda)^{n-1} =  n \lambda \wedge (d_\eta \lambda)^{n-1} = \iota_{Z_\lambda} (d_\eta\lambda)^n$$ does not vanish, hence $\ker\lambda$ is an even contact structure.
 Clearly,  $Z_\lambda$ positively spans the characteristic foliation.
  Put $\theta=-\eta$. Then, 
 $$d_\eta\lambda(Z_\lambda,R)=\iota_{Z_\lambda}(d_\eta\lambda)(R)=\lambda(R)=1$$
 $$(\eta\wedge\lambda)(Z_\lambda,R)=\eta(Z_\lambda)$$
 Hence,
 $\theta(Z_\lambda)+d\lambda(Z_\lambda,R)=1$ is positive;
while $\om_\theta = d_\eta\lambda$ (recall Equation (\ref{omega_eqn})). 

(ii) The vector field $Z$ spanning
positively the characteristic foliation $\ZZ$, the $1$-form $\lambda$ is $\eta$-Liouville
positive if and only if $\iota_R(\iota_Z( (d_\eta\lambda)^n))$ induces a positive form on $\varepsilon/\ZZ$.  This condition amounts to the positivity of the function $\chi_\lambda  - \eta(Z)$, since
$$\iota_R(\iota_Z( (d_\eta\lambda)^n))=
\iota_R(\iota_Z( d\lambda^n - n \eta \wedge \lambda \wedge d\lambda^{n-1}))
 = n \left(\chi_\lambda  - \eta(Z) \right) \wedge d\lambda^{n-1}$$
and since $d\lambda^{n-1}$ induces a positive volume form on $\varepsilon / \ZZ$.
  Finally, the value of $Z_\lambda$ results from the computation:
$$\iota_Z(d_\eta \lambda) = \LL_Z \lambda  - \iota_Z (\eta \wedge \lambda) = \left(\chi_\lambda - \eta (Z)\right) \lambda.$$
  \end{proof}

As a first application of Lemma \ref{derivate_lem} (ii), one has for  conformal symplectic structures an obvious  cut-and-paste method  along hypersurfaces of contact type  (Definition \ref{twisted_dfn_3}), generalizing the classical method for genuine symplectic structures.
The following lemma gives some precisions about these hypersurfaces.

\begin{lem}\label{hypersurface_lem}
 Let $\omega$ be a $\eta$-symplectic form on $M$, and
 $H\subset M$ be a cooriented hypersurface.

i) If $H$ is transverse to a $\eta$-Liouville vector field $Z$
for $\omega$, then
 $\lambda=\iota_Z(\om)$ is a  $(d_\eta)$-primitive of $\om$
near $H$, and $\alpha=\lambda\vert_H$
 is a contact form on $H$.

ii)   Conversely, if  $\om$ admits in restriction to $H$ a  $(d_\eta)$-primitive $\alpha$
which is
 contact,
  then
   $\alpha$ extends to a  $(d_\eta)$-primitive $\lambda$ of $\om$ on $\nb_M(H)$
   whose $\eta$-dual vector field $Z$ is
    transverse to $H$.
    
 iii)   Moreover, in (i) and (ii),
  $Z$ is positively (resp. negatively) transverse to $H$ iff $\alpha$
    is positive (resp. negative) as a contact form on $H$.
   \end{lem}
      We call $H$ of \emph{overtwisted} contact type if moreover $\ker\alpha$ is overtwisted on each connected component of $H$.
      \begin{proof} (i): after Lemma \ref{derivate_lem} (i).
      
      (ii):
   To get the extension $\lambda$,
   consider a tubular neighborhood $T$ of $H$ in $M$
   and a deformation retraction $h=(h_t):T\times\I\to T$ such that $h_0=\id_T$ and
   $h_1(T)=H$ and $h_t\vert_H=\id_H$.
  After the generalized
   Poincar\'e lemma \ref{poincare_lem} applied to $\omega$,
   $$e^{-F_1}h_1^*(\omega)-\omega=d_{\eta}\HH(\omega)$$
   On the other hand, after the ordinary Poincar\'e lemma applied to $\eta$,
   $$h_1^*(\eta)-\eta=dF_1$$
   so, using Remark \ref{cc_rmk}:
   $$e^{-F_1}h_1^*(\omega)=e^{-F_1}h_1^*(d\alpha-\eta\wedge\alpha)
   =e^{-F_1}(dh_1^*(\alpha)-h_1^*(\eta)\wedge h_1^*(\alpha))$$
   $$=e^{-F_1}d_{h_1^*(\eta)}h_1^*(\alpha)
   =d_\eta(e^{-F_1}h_1^*(\alpha))$$
   Hence,
   $$\lambda=e^{-F_1}h_1^*(\alpha)-\HH(\omega)$$
   is a ($d_\eta$)-primitive of $\omega$ on $T$; and coincides with $\alpha$
   in restriction to $H$, since $F_1$ and $\HH(\omega)$ vanish identically on $H$.
  
     Finally, $\lambda\vert_H=\alpha$ being positive (resp. negative)
 contact, after Lemma \ref{derivate_lem} (i),
      $Z$ is
    transverse to $H$.
    
    (iii) Obvious from Lemma \ref{derivate_lem} (i).
   \end{proof}

\section{Existence of conformal symplectic structures}\label{conformal_sec}
   
\begin{TA} On a closed connected manifold $M$ of dimension $2n\ge 2$, let $\eta$ be a closed, \emph{non-exact}
 $1$-form; and let $\omega$ be a nondegenerate $2$-form.
 
 Then, $\om$ is homotopic to a $\eta$-symplectic form,
  whose Novikov cohomology class in $H^2_{\eta}(M)$ may be prescribed.
\end{TA}

When moreover the cohomology class of $\eta$ is integral: $[\eta]\in H^1(M;\Z)$,
Eliashberg-Murphy have already
 obtained \cite{eliashberg_murphy_20} that, for some constant $c \neq 0$, the manifold $M$ admits a $(c\eta)$-symplectic form 
 homotopic to $\omega$ as a nondegenerate $2$-form.

We choose to prove Theorem $A$ with by
means of an auxiliary Morse function. One could, alternatively, use a handle decomposition;
 but the Morse function method
 adapts painlessly to the foliated framework (see Section \ref{cobordism_sec}).

\medskip
One begins
by endowing the compact solid torus  
$$T^{2n} = \S^1\times\D^{2n-1}$$ 
with a conformal symplectic structure inducing an overtwisted contact structure on its boundary,
either concave or convex. These conformal symplectic tori
will further down
 somehow play the role of ``symplectic cap''
 and ``symplectic cup'' in the proof of Theorem $A$. 

Denote by $\theta$ the pullback to $T^{2n}$ of the volume form on $\S^1$ with total volume one, and endow $T^{2n}$ with an arbitrary orientation.
  
\begin{lem}\label{toric_lem}
 For every integer $n\ge 1$, there exist a real constant $c_n$, and on $T^{2n}$
 a $1$-form $\lambda$ and a vector field $Z$ such that
\begin{enumerate}
\item $\theta(Z)=1$, and $Z$ exits or enters (at choice) $T^{2n}$
 transversely through $\partial T^{2n}$;
\item $\lambda$ is, for every real $c>c_n$, a positive $(-c\theta)$-Liouville form,
whose $(-c\theta)$-dual vector field is positively colinear to $Z$;
\item $\lambda$ restricts on $\partial T^{2n}$ to an overtwisted contact form.
\end{enumerate}
\end{lem}

\begin{proof} Fix a compact
collar neighborhood $A$
 of the boundary $\partial T^{2n}$ in $T^{2n}$; and a
 vector field $Z$ on $A$ such that
 $\theta(Z)=1$, and which exits or
enters $T^{2n}$ transversely through $\partial T^{2n}$, at choice.
The parallelizable solid torus
  $T^{2n}$ bears a 
positive nondegenerate $2$-form $\omega$.
Consider on $\partial T^{2n}$
 the induced almost contact structure $(\iota_Z(\omega)\vert_{\partial T^{2n}},
\omega\vert_{\partial T^{2n}})$ (Remark \ref{lwctctandnd}).
After the h-principle for overtwisted contact structures
\cite{borman_eliashberg_murphy_15},
  $\partial T^{2n}$ admits
 an overtwisted contact structure $\xi$ in the same formal homotopy
 class as $(\iota_Z(\omega)\vert_{\partial T^{2n}},
\omega\vert_{\partial T^{2n}})$.

Shrinking $A$ if necessary, let  $\varepsilon$ be the even contact structure on $A$
pullback of $\xi$ under the projection $A\to\partial T^{2n}$
along the flow lines of $Z$. 
Applying Lemma \ref{even_contact_lem} to the slice foliation of $T^{2n}$
 by the disks $t\times\D^{2n-1}$ ($t\in\S^1$), to $A$, and to $\omega$,
extend $\varepsilon$ to a global even contact
structure on $T^{2n}$, still called $\varepsilon$, whose characteristic foliation $\ZZ$ is
transverse to the disks,
and spanned
by  $Z$ on $A$. Extend $Z$ to a global vector field on $T^{2n}$,
 still called $Z$, spanning $\ZZ$, and
such that $\theta(Z)=1$.

Choose any $1$-form $\lambda$ on $T^{2n}$
 representing $\varepsilon$.
Let $c_n$ be the maximum of the function $-\chi_{\lambda}$ on $T^{2n}$.
For $c>c_n$, one has $(-c\theta)Z<\chi_{\lambda}$.
After Lemma \ref{derivate_lem} (ii), the $1$-form $\lambda$ is
 $(-c\theta)$-Liouville positive on $T^{2n}$, and its  $(-c\theta)$-dual vector field is
 positively colinear to $Z$.
\end{proof}

\begin{rmk}\label{positive_note} For $n > 1$,
 the constant $c_n$ given by Lemma \ref{toric_lem} is necessarily non-negative.
 Indeed, assume by contradiction that $c_n < 0$. Fix a negative $c>c_n$.
 Then, after properties (1) and (2) and after the proof of Lemma \ref{derivate_lem}, (i),
 the function $$d\lambda(Z_\lambda, R)=1-c\theta(Z_\lambda)$$ is $>1$ on $T^{2n}$. After Equation (\ref{chi_eqn}),
 the function $\chi_\lambda$ is positive on $T^{2n}$.
  So, $\lambda$ is a positive genuinely Liouville form on $T^{2n}$,
  whose corresponding Liouville vector field is positively colinear to $Z$.
  In particular, $Z$ has to exit the solid torus; and the genuinely symplectic form $d\lambda$
  fills the overtwisted contact structure $\xi$: a contradiction. 
  
  Consequently, Lemma \ref{toric_lem} lacks some symmetry: it
   would not hold if one changed
   $-c\theta$ to $c\theta$ in property (2). Equivalently,
    one cannot change $\theta (Z)=1$ to
  $\theta (Z)=-1$ in property (1). This lack of symmetry is
  linked to the unability of our method to
  deform foliations to contact structures on \emph{closed} manifolds, see Remark
  \ref{limitation_rmk}.
\end{rmk}

\begin{rmk}\label{ndhc} In Lemma \ref{toric_lem},
there is no need to explicitly prescribe the homotopy class of $d_\eta\lambda$
 among the positive nondegenerate $2$-forms on $T^{2n}$; indeed, there is only one such class,
 $\rm{SO}(2n)/\rm{U}(n)$ being simply connected.
\end{rmk}

\begin{proof}[Proof of Theorem $A$] 
 The closed $1$-form $\eta$ being not exact, $M$ contains
 two disjoint, embedded circles on which the integral of $\eta$
is less than minus the constant $c_n$ of Lemma \ref{toric_lem}.
 Thicken them into two disjoint ($2n$)-dim\-ens\-ional compact solid tori  $T_-, T_+\subset M$.
After Remark \ref{cc_rmk}, without loss of generality, we can
 change $\eta$ on $M$ to any cohomologous
closed $1$-form; in particular, we can arrange that
 on $T_\pm
\cong\S^1\times\D^{2n-1}$, the form $\eta$ is proportional
to $\theta$.

One endows $M$ with the orientation defined by $\omega^n$.
The lemma  \ref{toric_lem} provides, on $T_-$ (resp. $T_+$), a positive
$\eta$-Liouville form $\lambda_-$ (resp. $\lambda_+$) whose $\eta$-dual vector field $Z_-$
(resp. $Z_+$) enters (resp. exits) the torus transversely through its boundary,
on which the contact form  $\lambda_-\vert_{\partial T_-}$ 
(resp. $\lambda_+\vert_{\partial T_+}$)
 is overtwisted.
 
  We extend $\lambda_-$ and $Z_-$ (resp. $\lambda_+$
and $Z_+$) to a small open neighborhood of
$T_-$ (resp. $T_+$) in $W$.
 
 After Remark \ref{ndhc},
 $d_\eta\lambda_-$    (resp. $d_\eta\lambda_+$)  is homotopic to  $\omega$
 over $T_-$ (resp. $T_+$)
 as a nondegenerate $2$-form. After a homotopy of $\omega$,
 we arrange without loss of generality that $\omega=d_\eta\lambda_-$    (resp. $d_\eta\lambda_+$) over some neighborhood of $T_-$ (resp. $T_+$).

 On the complement $W=M\setminus Int(T_-\cup T_+)$, fix a Morse function $f$
 such that $f\mun(0)=\partial T_-$ and $f\mun(1)=\partial T_+$, and without local extrema
 in the interior of $W$.  Choose a descending pseudo-gradient $Z$ for $f$ on $W$,
 coinciding with $Z_-$ (resp. $Z_+$) close to $\partial T_-$
 (resp. $\partial T_+$).
 
  Every critical point $c$ of $f$ of index $1\le i\le 2n-1$
   admits in $W$ a small compact 
  neighborhood $H_c$  with (convex) cornered boundary, as follows. 
   $H_c$ is diffeomorphic to $\D^i\times\D^{2n-i}$ minus a
  small open
  tubular neighborhood of the corner $\S^{i-1}\times\S^{2n-i-1}$; and the boundary
  splits as
  $$\partial H_c=\partial_+H_c\cup\partial_0H_c\cup\partial_-H_c$$
 where 
  \begin{itemize} 
 \item
 $f$ is constant on $\partial_+H_c$ and on $\partial_-H_c$;
 \item $Z$ enters (resp. exits) $H_a$  transversely through $\partial_+H_c$
 (resp. $\partial_-H_c$);
 \item $Z$ is tangential to the $\I$ factor on
  $\partial_0H_c\cong\S^{i-1}
 \times\S^{2n-i-1}\times\I$.
 \end{itemize}
 
 Since $H_c$ is simply connected, we can
 without loss of generality arrange that $\eta=0$ on $\nb_W(H_c)$ (Remark \ref{cc_rmk}). Hence, in $H_c$, we actually look for a \emph{genuine} Liouville form.
  After the symplectization of cobordisms
  \cite{eliashberg_murphy_20},
   there is a positive Liouville form $\lambda_c$ on $\nb_W(H_c)$ whose
 dual Liouville vector field is positively colinear to
  $Z$ on $\nb_{W}(\partial H_c)$; and such that
 $\lambda_c$ restricts to an overtwisted
 contact structure on every connected component of $\partial_\pm H_c$.
 
Put for short $H=\cup_cH_c$, where $c$ runs over the critical points of $f$.
 After Lemma \ref{even_contact_lem} (ii) applied
in $W'=W\setminus Int(H)$ foliated by the level hypersurfaces of $f$, and
 $A=\partial W'$,
there is an even contact structure $\varepsilon$ on  $W'$ such that
\begin{itemize}
\item  The almost symplectic class associated with $\varepsilon$ contains $\omega$;
\item The characteristic
foliation $\ZZ$ of $\varepsilon$ is transversal to the level hypersurfaces of $f$;
\item $\varepsilon$ coincides respectively
 with the kernels
of $\lambda_-$, $\lambda_+$ and $\lambda_c$ on neighborhoods of
 $\partial T_-$, $\partial T_+$
and $\partial H_c$, for every critical point $c$.
\end{itemize}
 Changing the pseudo-gradient
 $Z$ in $Int(W')$,
one moreover arranges that $Z$ spans $\ZZ$ positively on $W'$.

By means of a partition of the unity, make a 1-form $\lambda$ on $M$
representing $\varepsilon$ on $W'$, matching $\lambda_+$
on $\nb_W(T_+)$, matching $\lambda_-$ on $\nb_W(T_-)$; and,
 for each critical point $c$, matching $\lambda_c$
on $\nb_W(H_c)$.

In particular, $\lambda$ is $\eta$-Liouville
 on some small open neighborhood $V$ of $T_-\cup T_+\cup H$
in $M$.

 \medbreak
 Claim 1: there is a smooth real function $g$ on $M$,
  locally constant on $T_-$, $T_+$ and $H$, such that
  $Z\cdot g<0$ on $Int(W')$.
  
Indeed, such a function $g$ will be obtained from $f$
 by a modification in arbitrarily small neighborhoods of
 $T_-$, $T_+$ and $H$. The modification is obvious close to $T_-$ and $T_+$.
 Now, consider a critical point $c$.  Write $t_-=f(\partial_-H_c)$
 and $t_+=f(\partial_+H_c)$.
 On a small enough open neighborhood $\Omega$ of $H_c$ in $W$,
one easily builds a smooth plateau function $\phi:
 \Omega\to[0,1]$ such that
 \begin{itemize}
 \item $\phi$ is compactly supported in $\Omega$, while
 $\phi\mun(1)=H_c$;
 \item $Z\cdot\phi(x)\ge 0$ (resp. $=0$) (resp. $\le 0$)
  at every point $x$ of $\Omega$ such that $f(x)\ge t_+$ (resp. $t_-\le f(x)\le t_+$)
  (resp. $f(x)\le t_-$).
 \end{itemize}
  Then, $g=(1-\phi)f+\phi f(c)$ works on $\Omega$. The claim 1 is proved.

  \medbreak
 After multiplying $g$ by a large enough positive constant, one arranges moreover
  that $Z\cdot g$
 is less than
  $\chi_\lambda-\eta Z$ on $M\setminus V$.
   \medbreak
Claim 2: the $1$-form $\mu=e^{-g}\lambda$
  is $\eta$-Liouville  on $M$.
  
This holds of course
 on $T_-$, on $T_+$ and on each $H_c$, since on these domains,
  $\mu$ is a {constant} multiple
 of $\lambda_-$, $\lambda_+$ and $\lambda_c$, respectively.
 On $W'$, in view of Lemma \ref{derivate_lem}, there remains to verify that
 $\chi_\mu=\chi_\lambda-Z\cdot g$ is more than $\eta Z$. But this inequality does hold on $M\setminus V$
 by choice of $g$; while
on $W'\cap V$, the function $\chi_\lambda-Z\cdot g$ is not less that $\chi_\lambda$,
which is more than 
   $\eta Z$
by Lemma \ref{derivate_lem}. The claim 2 is proved.
 \medbreak
 By construction, $d_\eta\mu$ lies over $W'$ in the almost symplectic class associated with $\varepsilon$. Globally, $d_\eta\mu$ is  homotopic to
 $\omega$ among the nondegenerate $2$-forms on $M$.

 Finally, in order to obtain a $\eta$-symplectic form $\omega'$ in any prescribed
  cohomology class $a \in H^2_{\eta}(M)$: first, fix an arbitrary $(d_\eta)$-closed $2$-form $\varpi$ on $M$ representing $a$.
 Second, define $$\omega' =  \varpi+K d_\eta\mu$$ for some large positive real constant $K$.
  Provided that $K$ is large enough,
  $\omega'$ is nondegenerate;
  and homotopic to $d_\eta\mu$, among the nondegenerate $2$-forms,
   through the homotopy $(1-t)\varpi+K d_\eta\mu$ ($t\in\I$) followed by the homotopy $
   ((1-t)K +t)d_\eta\mu$ ($t\in\I$).

\end{proof}

Theorem $A$ admits the following (easy) generalization, allowing a
smooth boundary
for the ambient manifold, and prescribing a natural boundary condition.
Let $M$ be
an oriented compact ($2n$)-manifold whose nonempty smooth boundary is splitted into two disjoint
compact subsets $\partial_\pm M$, each of which may be empty.

\begin{thm}\label{general_thm}
Assume that one is given on $M$
\begin{itemize}
\item A {nonexact} closed $1$-form  $\eta$;
\item A relative Novikov cohomology class $a\in H^2_\eta(M,\partial M)$;
\item A positive nondegenerate $2$-form $\omega$.

  \end{itemize}
  
 Then, there exist $\varpi\in\Omega^2(M)$ and $\alpha\in\Omega^1(\partial M)$ such that
 \begin{itemize}
 \item The pair $(\varpi,\alpha)$ is $D_\eta$-closed and
  represents the cohomology class $a$
 (recall Equation \ref{relative_eqn});
  \item  $\varpi$ is nondegenerate, and
  homotopic
  to $\omega$ among the nondegenerate $2$-forms on $M$;
 \item $\alpha$ is an overtwisted contact form on every
 connected component of $\partial M$, positive on $\partial_-M$
 and negative on $\partial_+M$. 
  \end{itemize}

\end{thm}

In particular, $\varpi$ is $\eta$-symplectic, and
 $\partial_+ M$ (resp. $\partial_-M$) is of concave (resp. convex)
  overtwisted contact type (Definition \ref{twisted_dfn_3} and Lemma \ref{hypersurface_lem})
  with respect to $\varpi$. 
(Our choice of signs, seemingly unnatural, is coherent with the pseudogradients
being descendant in section \ref{morse_sec}).

\begin{proof}[Proof of Theorem \ref{general_thm}] As in the proof of Theorem $A$,
one chooses two disjoint solid tori $T_\pm$ embedded in the interior of $M$,
on the cores of which the integral of $\eta$
is less than minus the constant $c_n$ of Lemma \ref{toric_lem}.
Then, $T_-$ (resp. $T_+$) bears a positive
$\eta$-Liouville form $\lambda_-$ (resp. $\lambda_+$)
 whose $\eta$-dual vector field
  enters (resp. exits) the torus transversely through its boundary,
on which the contact form  $\lambda_-\vert_{\partial T_-}$
 (resp. $\lambda_+\vert_{\partial T_+}$)
 is overtwisted.

 Fix on $\nb(\partial M)$ a vector field $Z$ transverse to $\partial M$,
entering (resp. exiting) $M$ through $\partial_+M$ (resp. $\partial_-M$).
The h-principle for overtwisted contact structures
\cite{borman_eliashberg_murphy_15} provides on $\partial M$
an overtwisted contact form $\beta$ in the same almost contact class
as $(\iota_Z(\omega)\vert_{\partial M},\omega\vert_{\partial M})$
(Definition \ref{ac_dfn} and Remark \ref{lwctctandnd}).
 By means of Lemma  \ref{derivate_lem},
extend $\beta$  over $\nb(\partial M)$ to a $\eta$-Liouville form
 $\lambda$,  $\eta$-dual to $Z$.

 On the complement $W=M\setminus Int(T_-\cup T_+)$, fix a Morse function $f:W\to [0,1]$
 such that $$f\mun(0)=\partial T_-\cup\partial_-M$$ $$f\mun(1)=\partial T_+
 \cup\partial_+M$$ and without local extrema
 in the interior of $W$. The same construction as in  the proof of Theorem $A$
 yields on $M$ a $\eta$-Liouville form $\mu$ which is  on $\nb(\partial M)$
 a positive locally constant multiple of $\lambda$;
 and $d_\eta\mu$ is homotopic to $\omega$
 as a nondegenerate $2$-form on $M$.

The relative Novikov cohomology class $a$ is represented by a
pair $(\omega',0)\in\Omega^2(M)\times\Omega^1(\partial M)$
such that $d_\eta\omega'=0$. 
  For a large enough positive real $K$, the pair $(\varpi,\alpha)$ obviously works, where 
   $$\varpi=\omega'+Kd_\eta\mu$$
   $$\alpha=K\mu\vert_{\partial M}$$
\end{proof}

\section{Morse theory for codimension-$1$ taut foliations}\label{morse_sec}

 See e.g. \cite{candel_00} for the elements on foliations. 
 Given a taut codimension-$1$ foliation,
the existence of functions which are Morse in restriction to every leaf
 is classical
\cite{ferry_wasserman_86}.
 The leafwise pseudo-gradients and their dynamics
appeared  in  \cite{bertelson_02}, for foliations of arbitrary codimensions,
in order to construct some leafwise geometric structures; and in
\cite{gay_18},
in order to study the contact forms carried by open book decompositions on $3$-manifolds.
 Apart from
these works, the
``Morse theory of foliations'' seems not to have
 met the attention that it deserves.

 In the present section, we
elaborate the tools
that we need for Section \ref{cobordism_sec}:
essentially,
 the construction of \emph{ordered} leafwise Morse functions
  (Definition \ref{ordered_dfn} below),
 and the cancellation of leafwise local extrema.

\medbreak
 In this section, we consider
  a compact manifold $M$ of dimension $m\ge 2$,
  maybe with a smooth boundary $\partial M$,
endowed with a cod\-im\-ension-$1$ foliation $\FF$, coorientable
to fix ideas, and transverse to $\partial M$.

For a smooth real function $f$ on $M$, a point $c\in M$ is \emph{leafwise critical}
if  the differential of $f$ vanishes on the leaf $L_c$ of $\FF$ through $c$.
The \emph{critical locus} $\crit(f,\FF)\subset M$ is the set of the leafwise critical points.

\begin{dfn}\label{lm_dfn} We call $f$ \emph{leafwise Morse} if
 $f$ restricts to a Morse function on every leaf of $\FF$,
and if $f$ is locally constant on $\partial M$. In particular, the critical locus is interior to $M$;
while $\partial M$ splits as the disjoint union of local minima $\partial_-(M,f)$
and of local maxima $\partial_+(M,f)$.
\end{dfn}

Informally, a leafwise Morse function locally amounts
to a $1$-parameter family of Morse functions in dimension $m-1$.

At every $c\in \crit(f,\FF)$, the Morse index $\ind_c(f,\FF)$ of $f\vert_{L_c}$ lies between $0$
and $m-1$.
Clearly, $\crit(f,\FF)$ is a disjoint union of circles  transverse to $\FF$,
and the index is constant on each circle.
One denotes by $\crit^i(f,\FF)$ the set of the index-$i$ leafwise critical points.
 Of course, in general $f$ is not locally constant on $\crit(f,\FF)$.
 Recall that
 
 \begin{dfn}\label{taut_dfn}
 $\FF$ is called \emph{taut} if every leaf meets a transverse loop.
 \end{dfn}
If $\F$ admits a leafwise Morse function, then
 $\FF$ must be taut, since every minimal set has to meet the index-$0$ critical locus
  and the index-($m-1$) critical locus. Conversely:
\begin{pro}[Ferry-Wasserman \cite{ferry_wasserman_86}]
\label{morse_pro} Let $M$ be a compact manifold with smooth boundary,
splitted into two compact subsets $\partial_\pm M$ (one of which may be empty,
or both).
Let $\F$ be on $M$ a coorientable, taut codimension-$1$ foliation,
transverse to $\partial M$.

Then, $\F$
admits a leafwise Morse function $f$ such that $\partial_-(M,f)=\partial_- M$
and $\partial_+(M,f)=\partial_+ M$.

\end{pro}

\subsection{Ordering a leafwise Morse function}\label{order_ssec}
Let $M$ be as before a compact manifold of dimension $m \ge 2$ with smooth boundary (maybe empty), endowed with a coorientable cod\-im\-ension-$1$, taut foliation $\F$ transverse to $\partial M$. 
Let $f$ be a leafwise Morse function on $M$.

\begin{dfn}\label{ordered_dfn} The leafwise Morse function
$f$ is \emph{ordered} if
 for every two leafwise critical points $c, c'$, the inequality
 $$\ind_c(f,\F)<\ind_{c'}(f,\F)$$ implies $f(c)<f(c')$.
\end{dfn}

\begin{pro}\label{ordered_pro} The foliation $\F$ admits an \emph{ordered}
 leafwise Morse function which has
 the same critical
locus as $f$, with the same indices, and the same splitting of $\partial M$
into local minima and local maxima.
\end{pro}
This will result from the generic dynamical properties of the leafwise pseudo-gradients.
 As is usual in Morse theory, one considers \emph{descending} pseudo-gradients.
 
 \begin{dfn}\label{pdg_dfn}
A vector field $\nabla$ on $M$,  tangential to $\F$,
 is a \emph{leafwise pseudo-gradient} for $f$
 if, in every leaf $L$, the restriction $\nabla\vert_L$
is a descending pseudo-gradient for the Morse function $f\vert_F$. In other words:
\begin{itemize}
\item The function $\nabla\cdot f$ is negative but at the leafwise critical points;
\item The Hessian of $\nabla\cdot f$ at every leafwise critical point $c$ is negative definite
in $T_cL_c$.
\end{itemize}
\end{dfn}
 The construction of such a vector field
is straightforward by means of a partition of the unity.
Clearly, $\nabla$ enters $M$ through
$\partial_+(M,f)$ and exits $M$ through $\partial_-(M,f)$.
Write $\nabla^t(x)$ for the image of $x\in M$
 under the flow of $\nabla$ at the time $t\in\R$,
whenever defined.

\begin{lem}\label{length_lem} (i) For every $x\in M$, the orbit $\nabla^t(x)$ descends from a
point $\alpha(x)\in \crit(f,\FF)\cup\partial_+(M,f)$ to a point
  $\omega(x)\in \crit(f,\FF)\cup\partial_-(M,f)$.
  
   (ii) Moreover,
the lengths
of the orbits have an upper bound not depending on $x$.
\end{lem}

\begin{proof} If not, some orbit
 would have infinite length; hence $f$ would not be bounded on $M$, a contradiction.
\end{proof}

\begin{dfn}\label{w_dfn}
 For every subset $X\subset\partial_-(M,f)\cup \partial_+(M,f)\cup \crit(f,\F)$,
define the stable and the unstable manifold of $X$ with respect to $\nabla$ as
 $$\W^s(\nabla, X)=\omega\mun(X)$$
 $$\W^u(\nabla,X)=\alpha\mun(X)$$
\end{dfn}

In particular, for each
 connected component $C$ of the  index-$i$ critical locus $\crit^i(f,\F)$ ($0\le i\le m-1$), 
 the stable (resp. unstable)
 manifold  $\W^s(\nabla, C)$
(resp.  $\W^u(\nabla, C)$) is a submanifold of $M$ transverse to $\F$ and to $\partial M$;
 and
 its interior is a bundle of fibre $\R^{m-1-i}$ (resp.
$\R^{i}$) over the circle $C$.

\begin{dfn} We say that a leafwise pseudogradient is \emph{globally Kupka-Smale} if,
 for every two connected components $C$, $C'$ of ${\rm Crit}(f,\F)$, the manifolds
$\W^s(\nabla, C)$ and  $\W^u(\nabla, C')$ are transverse in $M$.
\end{dfn}

 The global Kupka-Smale property is generic among the leafwise pseud\-o-grad\-ients for
 $f$.
Indeed, this genericity is well-known in the framework of  $1$-parameter families of functions;
a similar argument holds for leafwise Morse functions.
 \emph{From now on
 in this subsection \ref{order_ssec}, we assume
 that $\nabla$ is a globally Kupka-Smale leafwise pseudo-gradient. }
 
It is convenient to extend the index function by defining $\ind_x(f,\F)=-\infty$ for $x\in\partial_-(M,f)$ and $\ind_x(f,\F)=+\infty$ for $x\in\partial_+(M,f)$.

 \begin{lem}\label{monotony_lem}\
 For every
  $x\in M$, one has $\ind_{\omega(x)}(f,\F)\le\ind_{\alpha(x)}(f,\F)$.
   \end{lem}
\begin{proof} The only case to consider is when $x$ is not leafwise critical, but
 $c=\omega(x)$ and $c'=\alpha(x)$
are leafwise critical. Let $i$ (resp. $i'$) be the index of $c$ (resp. $c'$). Since
$\W^s(\nabla, c)$ and  $\W^u(\nabla, c')$ are transverse in $M$, and of respective
dimensions $m-i$ and $i'+1$, the dimension
of their intersection is $i-i'+1$. On the other hand,
 the intersection containing the orbit through $x$,
 its dimension is at least $1$.
\end{proof}

 \begin{rmk}  At this point, the Morse theory of foliations shows its limitations,
  which
  are well-known in the frame of the
  $1$-parametric families of Morse functions.  Call an orbit \emph{exceptional}
   if its extremities are two leafwise critical
 points
of {the same index.} Such an orbit corresponds to a handle sliding, in Smale's sense,
    in a $1$-parametric family of Morse functions.
     After the global Kupka-Smale property, $\nabla$
admits at most a finite number of {exceptional}
   orbits.
   
  The classical Kupka-Smale property does not in general hold in restriction to the leaves.
   In every leaf $L$,
    for every pair of critical points in $L$ \emph{but maybe a finite number,}
   their stable and unstable
  manifolds for the pseudo-gradient $\nabla\vert_L$
   are transverse \emph{in $L$.}
   
 In general, no choice of the pseudo-gradient can avoid
 the existence of exceptional orbits.  The extremities of an exceptional
orbit can belong to the same connected component of
the critical locus. Also, given  two components $C$, $C'$ of the critical locus
 of the same index, there may exist two exceptional orbits, the one from $C$ to $C'$
 and the other from $C'$ to $C$. Then,
 $C$ and $C'$  cannot be separated in $M$ by any hypersurface
transverse to $\nabla$.
\end{rmk}

However, if $C$ and $C'$ have distinct indices, such a separating hypersurface does exists.

\begin{lem}\label{separation_lem}
 For each $1\le i\le m-1$, there is a closed hypersurface $H$ in $M$,
transverse to $\nabla$ and splitting $M$ into two domains $M_-$, $M_+$,
 such that
 \begin{itemize}
 \item 
   $\partial_-(M,f)$ and the leafwise critical points of indices 
at most $i-1$ lie in $M_-$;
\item $\partial_+(M,f)$ and the leafwise critical points of indices 
at least $i$ lie in $M_+$.
\end{itemize}
\end{lem}

\begin{proof} The proof 
belongs to elementary general topology. For short,
put $$M'=M\setminus \crit(f,\F)$$ Consider
 the space $\OO$ of the orbits of $\nabla$ which are regular 
 (in other words, not reduced to a single
  leafwise critical point),
 endowed
  with the quotient topology;
and the projection $\pi:M'\to\OO$.  For short, write $\ind(x)$ instead of $\ind_x(f,\F)$.
\medbreak
Claim 0 ---
\emph{ The map $$M\to\{0,1, \dots, m-1,+\infty\}:x\mapsto\ind(\alpha(x))$$
is lower semicontinuous on $M$.  The map $$M\to\{-\infty,0,1, \dots, m-1\}:x\mapsto\ind(\omega(x))$$
is upper semicontinuous on $M$.}

This follows at once from Lemma \ref{monotony_lem}.
\medbreak
Claim 1 --- \emph{  $\OO$  is a smooth
 $(m-1)$-manifold  \emph{in general not Hausdorff,} compact
 (in the sense that $\OO$ has the usual open cover property) and without boundary.}
 
  Indeed, for every $t\in\R$ and every embedding $\phi$
  of the open $(m-1)$-disk $D^{m-1}$ into the level set $M'\cap f\mun(t)$,
  the image $\phi(D^{m-1})$ meets every orbit at most once (since $\nabla\cdot f\le 0$);
  hence $\pi\circ\phi$ is a local coordinate chart for $X$. One thus gets an atlas
  whose changes of
  coordinates are obviously smooth. Clearly, there is a small open neighborhood
  of $\crit(f,\FF)$ in $M$ whose complement meets every regular orbit;
   hence $\OO$ is compact. The claim 1 is proved.
 \medbreak
Let us understand the lack of Hausdorff separation in $\OO$.
 By an \emph{orbit chain}, we mean  a finite sequence of regular
orbits
$\pi(x_1)$, \dots, $\pi(x_\ell)$ ($\ell\ge 1$)
 such that
  $\omega(x_{j-1})=\alpha(x_j)$ for each $2\le j\le\ell$.
 The \emph{endpoints} of the chain are the pair $(\alpha(x_1),\omega(x_\ell))$.
 After Lemma \ref{monotony_lem}, the indices of the critical points of the
 chain form a nonincreasing sequence \emph{(monotony property)}.
 \medbreak
 Claim 2 ---\emph{ From any sequence of regular
  orbits, one can extract a subsequence Hausdorff-converging
 towards an orbit chain.}
 
 This follows easily from Lemma \ref{length_lem} (ii).
 \medbreak
Claim 3 --- \emph{If
 two distinct regular orbits
$\pi(x)$, $\pi(y)$ are {note} separated in $\OO$, then they belong to a same orbit chain.}

This follows at once from Claim 2 applied to a sequence of regular orbits $\pi(x_k)$ which
converges both to $\pi(x)$ and to $\pi(y)$ in $\OO$.
\medbreak
Consider the subset $U\subset M'$ of the noncritical
points $x$ such that
\begin{equation}\label{U_eqn}
\ind(\alpha(x))\ge i\text{ and }
\ind(\omega(x))\le i-1
\end{equation}
Since $U$ is open in $M'$ (Claim 0),  $\pi(U)$ is open in $\OO$.
By Claim 3 and the monotony property,
 $\pi(U)$  is Hausdorff.
\medbreak
Claim 4 --- \emph{$\pi(U)$ is compact.}

Indeed, let $(u_k)$ be a sequence in $U$. Following Claim 2, after passing to a subsequence,
 the orbits through $u_k$ Hausdorff-converge in $M$ to an orbit chain.
  The index function being
 locally constant on $\partial M\cup \crit(f,\F)$, the endpoints $(c,c')$
 of this chain satisfy $\ind(c)\ge i$ and $\ind(c')\le i-1$. By the monotony
 property, one of the orbits $\pi(x)$ composing the chain satisfies 
 the inequalities (\ref{U_eqn}).
 We have thus found an accumulation point $\pi(x)$ in $\pi(U)$
 for the sequence $\pi(u_k)$:
the claim 4 is proved.

\medbreak
To sum up, $\pi(U)$ is a Hausdorff closed $(m-1)$-manifold.
 The restricted projection $\pi\vert_U$
is a submersion of $U$ onto $\pi(U)$ whose fibres are all diffeomorphic to $\R$.
Such a projection necessarily
 admits a section $s$: one can solve this elementary exercise,
or alternatively apply a more general
lemma due to Haefliger, see e.g. \cite{meigniez_02}.
The image $H=s(\pi(U))\subset U$ is a closed hypersurface transverse to $\nabla$, and
separating $U$ into two domains $U_-$, $U_+$ such that $\nabla$
enters $U_-$ and exits $U_+$ along $H$. Let $M_-$ (resp. $M_+$)
be the topological closure of $U_-$ (resp. $U_+$) in $W$. 
\medbreak
Claim 5 --- \emph{ For every
 $x\in M\setminus U$, one has
 \begin{itemize}
 \item $x\in M_-$ iff $\ind(\alpha(x))\le i-1$;
 \item $x\in M_+$
 iff $\ind(\omega(x))\ge i$.
 \end{itemize}}
 
 Indeed, the points $y\in M$
 such that $$\ind(\alpha(y))\ge m-1\text{ and }\ind(\omega(y))\le 0$$
  form an open (after Claim 0) and dense subset
 in $M$. Hence, $x$ is the limit of a sequence $(y_p)$ of such points. For $p$
 large enough, $y_p\in U$, hence
 its orbit $\pi(y_p)$ intersects transversely
  $H$ in a unique point $h_p=s(\pi(y_p))$. The point $h_p$ splits the orbit
  $\pi(y_p)$ into two subintervals  $\pi(y_p)\cap U_-$ and  $\pi(y_p)\cap U_+$.
   By Claim 2, after passing to a subsequence,
 the orbits $\pi(y_p)$ Hausdorff-converge in $M$ to an orbit chain 
 $\pi(x_1)$, \dots, $\pi(x_\ell)$, such that $x=x_j$ for some $1\le j\le\ell$.
  The index function being
 locally constant on $\partial M\cup \crit(f,\F)$, the endpoints of the chain satisfy 
 $$\ind(\alpha(x_1))\ge m-1
 \text{ and }\ind(\omega(x_\ell))\le 0$$
  By the monotony property, one and only one
   of the orbits $\pi(x_k)$ in
  the chain lies in $U$. We can choose $x_k\in H$.
  The point $x_k$ splits the orbit chain
  into two subintervals which are the Hausdorff limits of
   $\pi(y_p)\cap U_-$ and  $\pi(y_p)\cap U_+$; the first is thus contained in $M_-$,
   the second in $M_+$.
  Since $x\notin U$, one has $j\neq k$.
 If $j<k$, then $\ind(\omega(x))\ge i$ (monotony property)
  and $x\in M_+$. If $j>k$, then $\ind(\alpha(x))\le i-1$ (monotony property)
 and $x\in M_-$. The proofs of Claim 5 and of
 Lemma \ref{separation_lem} are complete.
\end{proof}

\begin{proof}[Proof of Proposition \ref{ordered_pro}]  For each $1\le i\le m-1$, after 
Lemma \ref{separation_lem}, there is a smooth plateau function $\phi_i$ on $M$
such that
\begin{itemize}
\item
 $\nabla\cdot\phi_i\le 0$;
\item $\phi_i=0$ on a neighborhood of $\partial_-(M,f)$ and of $\crit^{\le i-1}(f,\F)$;
\item $\phi_i=1$ on a neighborhood of
$\partial_+(M,f)$ and of  $\crit^{\ge i}(f,\F)$.
\end{itemize}
For every $\epsilon>0$, consider on $M$ the function
 $$g_\epsilon={\phi_1+\dots+\phi_{m-1}+\epsilon f}$$
 Obviously, every point of $\partial M$ is a local extremum for $g_\epsilon$; one has $\partial_\pm(M,g_\epsilon)=\partial_\pm(M,f)$;
 and $g_\epsilon$ coincides, for each $i$,  with
 ${i+\epsilon f}$ on some neighborhood of $\crit^i(f,\F)$.
Moreover, $\nabla\cdot g_\epsilon<0$ on $M\setminus \crit(f,\F)$.
Hence, $g_\epsilon$ is leafwise Morse with the same critical locus and the
same indices as $f$. Clearly,  $g_\epsilon$ is ordered
  provided that $\epsilon\vert f\vert<1/2$ on $M$.
\end{proof}

We end this subsection with a corollary of Proposition \ref{ordered_pro}.
 Let $f$ be an ordered leafwise Morse function on $M$.
 
 \begin{dfn}\label{nsi_def}  We call $f$ \emph{nearly self-indexing} if
 \begin{itemize}
\item $\vert f(c)-\ind(c)\vert<1/2$  at every leafwise critical point $c$;
\item $f\mun(-1/2)={\partial_-(M,f)}$ and $f\mun(m-1/2)={\partial_+(M,f)}$.
\end{itemize}
\end{dfn}
 
Recall that $\partial_-(M,f)$ and/or $\partial_+(M,f)$ can be empty.
  It is convenient, after reparametrizing
the values of $f$, to arrange that $f$ is nearly self-indexing.
We have thus
 decomposed $M$ into $m$
compact domains $$M_i=f\mun[i-1/2,i+1/2]$$
($0\le i\le m-1$)
with boundaries transverse to $\F$. Write $M_{\le i}$ for $ f\mun([-1/2,i+1/2])$
and $M_{\ge i}$ for $f\mun([i-1/2,m-1/2])$.

\begin{rmk} This
 is not quite a handle decomposition.
In every leaf $L$, and for each $i$, the generally noncompact manifold $L\cap M_i$
 decomposes into a countable, locally finite family of index-$i$ compact
handles.
 When one moves continuously from one leaf to another, at each exceptional orbit,
one of  these
handles slides on another.
\end{rmk}

\begin{dfn}\label{faithful_dfn} A closed submanifold $S\subset M$ is said to be \emph{faithful} to $\F$ if $S$ is transverse to $\F$, meets every leaf $L$ of $\F$, and if
the intersection $L\cap S$ is connected (hence a single leaf of $\F\vert_S$).
\end{dfn}
In other words, the embedding $S \subset M$ induces a Haefliger equivalence between the holonomy pseudogroups of the foliations $\F\vert_S$ and~$\F$.

\begin{cor}\label{faithful_cor} Let $f$ be a nearly self-indexing leafwise Morse function on the $m$-dimensional foliated manifold $(M,\F)$. If $m\ge 4$, then for each $1 \le i \le m-3$, the level set $f\mun(i+\frac12)$ is faithful to $\F$.
\end{cor}

\begin{proof} Consider a leaf $L$ of $\F$ and, on this connected  ($m-1$)-manifold maybe
with boundary, the genuinely Morse function  $g=f\vert_L$.
In $\R$, the value $i + \frac12$ does not lie in the closure of $g(\partial L)$,
and separates the critical values of $g$ of indices $0$ and $1$ from the critical
values of $g$ of indices $m-2$ and $m-1$.
Hence, $g^{-1}(i + \frac12)$
 is connected. 
\end{proof}

By induction on $m$, there also exists a closed $3$-dimensional submanifold of $M$ faithful to $\F$. The existence of faithful hypersurfaces in every taut codimension-$1$ foliated manifold of dimension at least $4$ is already known \cite{meigniez_16}, but we feel that the present construction, by means of the Morse theory of foliations, is more natural and clearer. See also \cite{martinez_14} for some particular cases of faithful submanifolds, obtained by different means.

\subsection{Cancelling leafwise local extrema}\label{cancellation_ssec}

Let, as before, $M$ be a compact $m$-dimensional manifold with smooth boundary, endowed with a codimension-$1$ coorientable taut foliation $\F$, transverse to $\partial M$.
Let $\partial_-M$, $\partial_+ M$ be a partition of $\partial M$ into two open subsets
(perhaps empty).

\begin{pro}\label{cancellation_pro} Assume that $m\ge 5$ and that every leaf of $\F$ meets $\partial_-M$ (\rp both $\partial_-M$ and $\partial_+M$).

 Then, $M$ admits a nearly self-indexing leafwise Morse function $f$ without  leafwise local minima (\rp extrema) in the interior of $M$ and such that $\partial_\pm(M,f)=\partial_\pm M$.
\end{pro}

\begin{rmk} Conversely, the existence of a function $f$
 such that  $\partial_\pm(M,f)=\partial_\pm M$ and without leafwise local
minima (resp. extrema) in the interior implies that every leaf $L$ meets $\partial_-M$ (resp.
both $\partial_-M$ and $\partial_+M$), since
the topological closure $\bar L$ of $L$ in $M$ being compact and saturated, 
the restricted function $f\vert_{\bar L}$
 must reach a minimum and a maximum.
 \end{rmk}
 \begin{rmk}  We do not know if Proposition \ref{cancellation_pro} holds as well
 for $m=3$ nor $4$.
\end{rmk}

\begin{rmk} When $\partial_-M = \emptyset$, the conclusion of Proposition \ref{cancellation_pro} amounts to say that $\F$ is, in the interior of $M$, ``uniformly open'' in the sense of~\cite{bertelson_02}. Hence, any leafwise open invariant differential relation abides the h-principle. In particular, the parametric h-principle holds for leafwise symplectic structures on such manifolds. Unfortunately, that $h$-principle does not allow any control over the structure at the boundary of $M$, unlike Theorem $C$.
\end{rmk}

In order to prove Proposition \ref{cancellation_pro},
we start with a nearly self-indexing leafwise Morse function $f$ on $M$ such that
 $\partial_\pm(M,f)=\partial_\pm M$
(Propositions \ref{morse_pro} and \ref{ordered_pro}). 
 We assume that every leaf meets
$\partial_-M$, and
  we shall cancel the interior leafwise local minima
  of $f$. Of course, if moreover every leaf meets also
$\partial_+M$, a symmetric method also cancels the interior leafwise local maxima.

The cancellation method is inspired by Laudenbach's
 reproof \cite{laudenbach_14}  of the classical Cerf cancellation lemma
  for local extrema in $1$-parameter families of functions. We refer to
 Laudenbach for some details. However, our foliated framework also calls for some
 specific arguments.

 We shall see how to cancel one connected component $C$ of theindex-$0$ critical locus.
  Repeating this argument removes  all the components. After each step,
  $f$ is not ordered any more, but
  we apply Proposition \ref{ordered_pro} and reorder $f$. 
 {  The steps of the cancellation of $C$
  are represented on the (somehow round) Cerf
 diagrams of Figure \ref{cerf_fig}.}
 
 Fix for $f$  a leafwise descending pseudo-gradient $\nabla$ which is globally Kupka-Smale.
 Recall the notations  $\omega(x)$,
  $\W^s(\nabla, X)$ {
   from the above Lemma \ref{length_lem} and Definition \ref{w_dfn}}.
   
The level set $S = f^{-1}(3/2)$ is important in the proof. This compact hypersurface of $M$ is
  transverse to $\F$ and to $\nabla$, and
 separates the leafwise critical points of $f$ of index $0$ and $1$
  from the leafwise critical points of indices $\ge 2$. For every connected
  component $C'$ of $\crit^1(f,\F)$, the intersection $S\cap\W^s(C')$
  is in $S$ a hypersurface  (in fact a bundle of fibre $\S^{m-3}$ over the circle)
   transverse to
  the foliation $\F\vert_S$.
  The  intersections of $S$ with the
   stable manifolds of
  $\partial_-M$ and of $\crit^0(f,\F)$ are, in $S$, finitely many open domains separated by
  these hypersurfaces.

 \begin{lem}\label{stable_lem_1} The endpoint map $\omega$ restricted to $S$ admits
 over the circle $C$ a smooth section $\sigma:C
 \to S$.
 \end{lem}
 
 \begin{proof} The bundle map $\omega:\W^s(\nabla,C)\to C$
  has fibre $\R^{m-1}$. Since $m\ge 4$, a generic
 section $s$ of this bundle is disjoint from $C$ and from
   the $2$-dimensional
 unstable manifolds of the index-$1$ critical locus. 
 For every $c\in C$,
 the orbit of $\nabla$ through $s(c)$ enters $\{f\le 3/2\}$
 at a unique point $\sigma(c)\in S$.
 \end{proof}

Note that $\sigma$ is transverse {in $S$ to the foliation} $\F\vert_S$.

 \begin{lem}\label{stable_lem_2}
  Every leaf of $\F\vert_S$ meets the stable manifold
  $\W^s(\nabla,\partial_-M)$.
 \end{lem}
\begin{proof}
Fix a point $x\in S$; let $L_x$ be the leaf of $\F$ through $x$ in $M$. On the one hand,
 $\nabla$ exits $M$ through $\partial_-M$ which meets $L_x$ by hypothesis.
  On the other hand, in
 $L_x$ which is of dimension $\ge 2$,
 the unstable manifolds of the index-$1$ critical points of $f\vert_{L_x}$
  form only a denumerable (and even locally finite)
  family of $1$-dimensional
 orbits. Hence,
 $\W^s(\nabla,\partial_-M)\cap L_x$ intersects $S$ in a point $y$. Finally,
 $x$ and $y$ both lie on $L_x\cap S$, which is a single leaf of $\F\vert_S$
 since the level set $S$ is faithful (Corollary \ref{faithful_cor}).
 \end{proof}

\begin{proof}[Proof of Proposition \ref{cancellation_pro}]

 For every $c\in C$ and $v\in\R$, by a \emph{bridge arc} for $c$ at level $v$, we mean
  an embedding $a$ of the interval $\I=[0,1]$ into $f\mun(v)$,
   tangential to $\F$, such that $\omega(a(0))=c$
   and $\omega(a(1))\in \partial_-M$.
    We say that the bridge arc $a$ is \emph{pointed}
    if moreover $v=3/2$ and $a(1/2)=\sigma(c)$.
   
   We are first interested in the level
   $v=3/2$. For every $c\in C$, a 
   pointed 
   bridge arc for $c$ exists by Lemma 
 \ref{stable_lem_2}.
 Then, pushing this
 arc into the neighborhing leaves  of $\F\vert_S$, one obtains,
 over some small open neighborhood $V_c$ of $c$ in $C$,
 a choice of a pointed 
 bridge arc for every $c'\in V_c$, depending smoothly on $c'$.
  Identifying $C$ with $\R/\Z$,
  let $\delta>0$ be a Lebesgue number for the open cover $(V_c)$.
  Choose a subdivision of the circle $C$ into $\ell$ intervals $[c_{i-1},c_{i}]$
 ($1\le i\le\ell$, $c_0=c_\ell$) of length less than $\delta$. Then, fix
  $\epsilon>0$ so small that for each $i$: $$2\epsilon<
  \vert c_{i}-c_{i-1}\vert<\delta-2\epsilon$$
  Put $I_i=[c_{i-1}-\epsilon,c_{i}+\epsilon]\subset C$. Choose for every $c\in I_i$
  a pointed 
  bridge arc $t\mapsto a_i(c,t)$, the map $a_i$ being a smooth
 embedding $I_i\times\I\hookrightarrow S$.

For every subset $X\subset M$, let $\nabla^+(X)$ denote the set of the points
$\nabla^t(x)$ for $x\in X$ and $t\ge 0$ (wherever defined).
 The cancellation method will modify $f$ and $\nabla$,  in $M$,
close to the squares $a_i(I_i\times\I)$
and  $\nabla^+(a_i(I_i\times 0))$
and  $\nabla^+(a_i(I_i\times 1))$ ($1\le i\le\ell$).
Since these squares can intersect each other for different values of $i$,
 we take some previous cautions so that the modifications don't interfer with each other.
Fix a small compact interval $J\subset\R$ centered at $3/2$, so small that $J$ contains
 no leafwise critical value of $f$; fix $\ell$ 
 values $v_i\in J$ ($1\le i\le\ell$), two by two distinct;
 let $\pi_i:f\mun(J)\to f\mun(v_i)$ be the projection along the flowlines of $\nabla$;
 let $\tilde a_i=\pi_i\circ a_i$.
Thus, for each $i$:
  \begin{itemize}
\item For every $c\in I_i$, the arc $t\mapsto \tilde a_i(c,t)$ is
a bridge arc for $c$ at level $v_i$;
\item The arc $I_i\to f\mun(v_i):c\mapsto\tilde a_i(c,1/2)$ extends to a global section
$\sigma_i:C\to f\mun(v_i)$ of $\omega$ over $C$.
\end{itemize}
(Namely, $\sigma_i=\pi_i\circ\sigma$).
These two properties are stable by any small enough isotopy of the embedding
$\tilde a_i$ in $f\mun(v_i)$ tangentially to $\F\vert_{f\mun(v_i)}$.
  Since $m\ge 5$,  after a generic such perturbation for each $i$,
 we can arrange that
   for every $1\le i<j\le\ell$ and $c\in I_i\cap I_j$, 
   the two flow lines $\nabla^+(\tilde a_j(c\times\partial\I))$ are disjoint from
   the arc $\tilde a_i(c\times\I)$.

Next, for {each $1\le i\le\ell$,}
 modify $f$ in a small neighborhood of $\sigma_i(C)$ in $M$
 by introducing,  for every $c\in C$, in the function $f\vert_{L_c}$,
 close to $\sigma_i(c)$ in the leaf $L_c$, and
 slightly above the level $v_i$,
a pair of critical points $s_i^1(c)$, $s^i_2(c)$ of respective indices $1$, $2$,
in cancellation position.  (Figures \ref{creation_fig} and \ref{cerf_fig} (a))

\begin{figure}
\includegraphics*[scale=0.5, angle=-90]{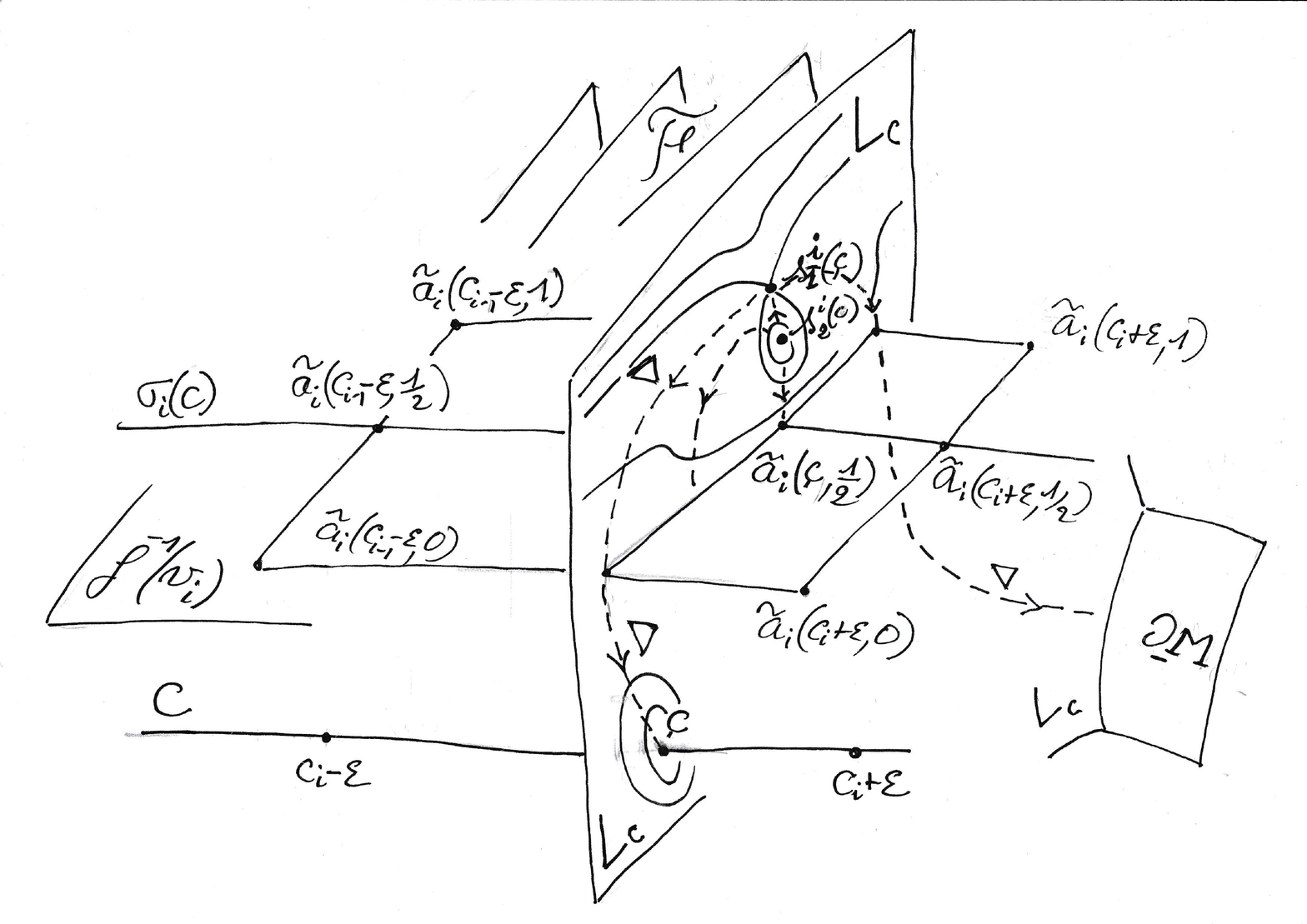}
\caption{
 Creation of two circles of leafwise critical points of respective indices $1$, $2$.
Beware that the ambiant dimension $m=3$, which this figure evokes, is excluded in the text.
Of course, for $m\ge 4$, the point $s_2^i(c)$ is not a local extremum in the leaf $L_c$.}
\label{creation_fig}
\end{figure}

\begin{figure}
\includegraphics*[scale=0.45, angle=-90]{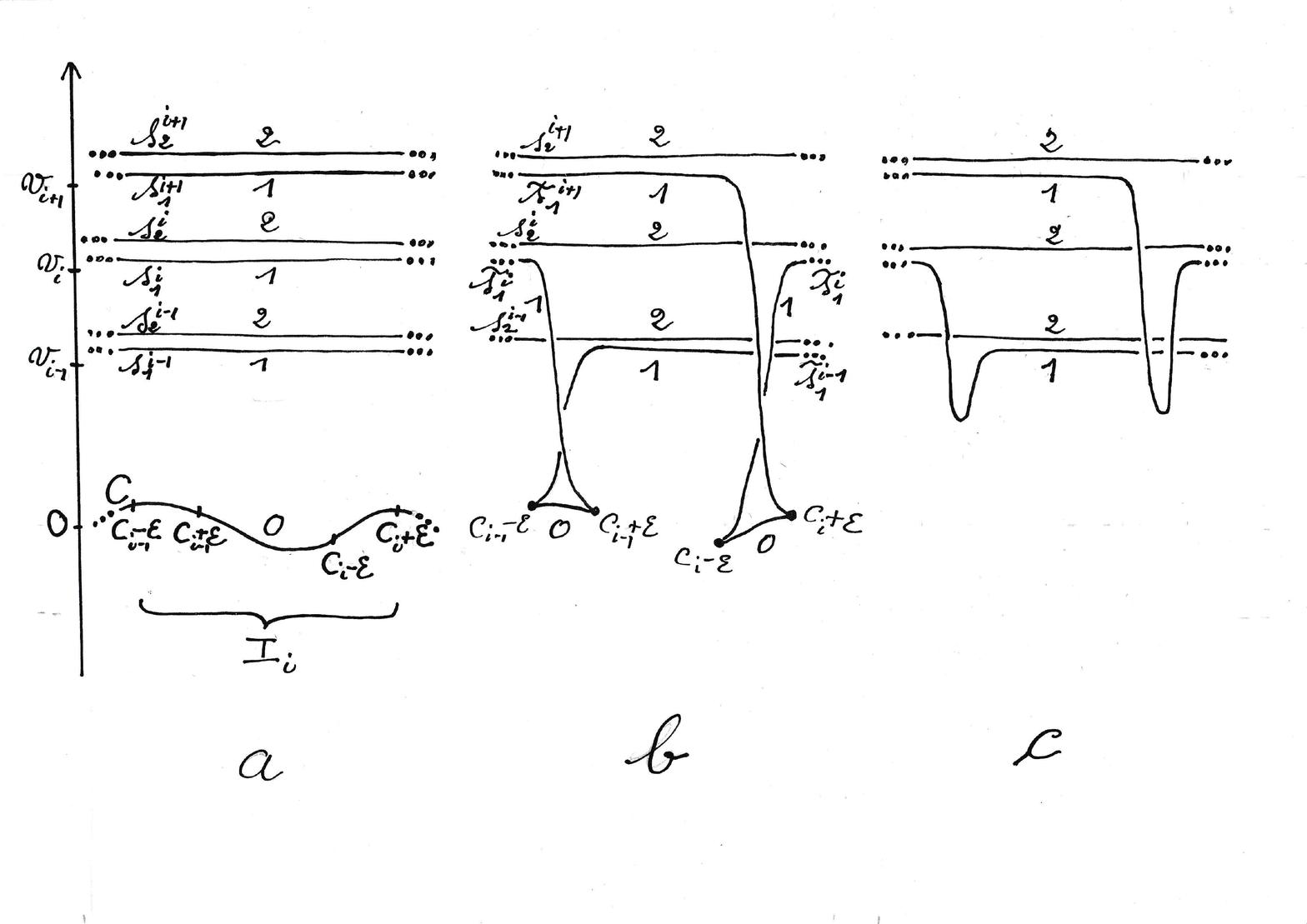}
\caption{
A somehow round Cerf diagram for the cancellation of leafwise local extrema.}
\label{cerf_fig}
\end{figure}

  For a suitable new leafwise pseudo-gradient (still denoted by $\nabla$),
 for every $c\in I_i$
  one has   (Figure \ref{creation_fig})
   $$\W^u(\nabla,s^i_1(c))\cap f\mun(v_i)=\tilde a_i(c\times\partial\I)$$
  $$\W^u(\nabla,s^i_2(c))\cap  f\mun(v_i)=\tilde a_i(c\times Int(\I))$$
  In particular, one of the two branches of the unstable manifold of $s^i_1(c)$
   descends to $c$, and the other descends to $ \partial_-M$.

   Since moreover the values of $f$ on $\partial_-M$ are less than the values of $f$ on $C$, 
    the \emph{parametric Morse cancellation lemma}
  applies to the pairs of {critical points} $(c,s^i_1(c))$  for $c\in[c_{i-1}+\epsilon,
  c_i-\epsilon]$.
   As a result, the function $f$ is modified in a small neighborhood
   of the square $$\W^u(\nabla,s^i_1([c_{i-1}+\epsilon,
  c_i-\epsilon]))\cup[c_{i-1}+\epsilon,
  c_i-\epsilon]$$ 
    so that the pairs $(c,s^i_1(c))$ are cancelled;
     but we may arrange that $f$ remains unchanged close to $\partial_-M$.
     
          Once these cancellations have been performed for every $1\le i\le\ell$,
     the resulting function is
 of course not leafwise Morse any longer.  
   On the way, at the point  $c_i - \epsilon$ (resp.  $c_{i-1} + \epsilon$),
    a birth (resp. death) critical point has been created in its leaf;
   so that,  instead of the original index-$0$ leafwise critical
 circle $C$, one now has  (see Figure \ref{cerf_fig} b) for each $i=1,\dots,\ell$~:

 \begin{itemize}
  \item An open arc $(c_i-\epsilon,c_i+\epsilon)$ of index-$0$ leafwise
  critical points;
  \item An open arc of leafwise index-$1$
  critical points $\tilde s^i_1(c)\in L_c$
   ($c\in(c_i-\epsilon,c_{i-1}+\epsilon)$),
  such that $\tilde s^i_1(c)=s^i_1(c)$ on a neighborhood of $[c_i+\epsilon,c_{i-1}-\epsilon]$;
  \item An index-$2$ leafwise
  critical circle $s^i_2(C)$.
  \end{itemize}

 The resulting Cerf diagram looks like Figure \ref{cerf_fig} (b),
with $\ell$ swallowtails.

Once the leafwise pseudogradient has been modified accordingly,
 one of the two branches of the unstable manifold of
 $\tilde s^i_1(c)$ descends to $c$, for every $1\le i\le\ell$ and
  $c\in(c_i-\epsilon,c_{i-1}+\epsilon)$.
Hence, for each $1\le i\le\ell$, and for
 $c\in[c_i-\epsilon,c_i+\epsilon]$,
  the family of triples of leafwise critical points $c$, $\tilde s^i_1(c)$,  $\tilde s^{i+1}_1(c)$
  of respective indices $0$, $1$, $1$
 matches the hypotheses of the \emph{elementary swallowtail lemma} 
  (Lemma 3.5
 of \cite{laudenbach_14}).
 We modify the function according to this lemma: the index-$0$ leafwise critical points
 vanish.  (Figure \ref{cerf_fig}, (c)).
  
  Once the $\ell$ swallowtails have been so cancelled, the resulting function is
  leafwise Morse again, and has got, instead of
   the original circle $C$ of index-$0$ leafwise critical points,
   $\ell$ more circles of index-$2$ leafwise critical points,
   and one more circle (covering  $\ell$ times the circle $C$) of index-$1$ leafwise critical points. The proof of Proposition \ref{cancellation_pro} is complete.
\end{proof}

\subsection{Orienting the stable and unstable manifolds}\label{orient_ssec}

We shall need in Section \ref{cobordism_sec} yet another normalization lemma
falling to the Morse theory of foliations. Consider, on an orientable manifold $M$ of dimension $m\ge 3$, a coorientable codimension-$1$ foliation $\F$, a leafwise Morse function $f$, and a connected component $C$ of the critical locus of $f$, of some nonextremal index $1\le i\le m-2$. 

\medskip
There are of course, up to isomorphism, exactly two dimension-$i$ real vector bundles
 over the circle~: the trivial one and the nonorientable one. The unstable manifold $\W^u(C)$
  with respect to any leafwise pseudo-gradient $\nabla$ is isomorphic to one of them,
  not depending on the choice of $\nabla$. Since $\F$ is tangentially orientable, the
   orientability of $\W^u(C)$ is equivalent to that of $\W^s(C)$.

\begin{pro}\label{orientation_pro} Assume that the stable and unstable manifolds of $C$ are not orientable.

Then, there exists a leafwise Morse function $g$ on $M$, coinciding with $f$ outside some arbitrarily small neighborhood $T$ of $C$, and whose leafwise critical locus in $T$ consists of two critical circles, one of index $i$ and one of index $i+1$ (or $i-1$ if one prefers); both of which have orientable stable and unstable manifolds.
\end{pro}

\begin{proof} We prove the $i+1$ case. Replacing $f$ by $-f$ implies the $i-1$ case.

Let $x$, $y$, $z_1$, \dots, $z_{m-3}$ denote the standard coordinates on $\R^{m-1}$. Consider the quadratic form $q(x,y) = (-x^2+y^2)/2$ on $\R^2$, and on $\R^{m-3}$
 a nondegenerate quadratic form $Q$ of index $i-1$. Let $R^\theta$ denote the rotation of angle $\theta$ in the $(x,y)$-plane.

Since $M$ is orientable but the stable and unstable manifolds of $C$ are not, the \emph{parametric Morse lemma} implies that $C$ admits a compact tubular neighborhood $T\cong
(\R/\Z)\times\D^2\times\D^{m-3}$ in $M$ on which $f$ has the form~:
$$f(t,x,y,z) = f(t,0,0,0) + q(R^{t\pi}(x,y)) + Q(z).$$
Let $b : \R \to [0,1]$ be a smooth even function with support $[-1,1]$ such that $b(y)=1-y^2/2$ near $y = 0$ and the derivative $b'$ is negative on $(0,1)$ and has a unique critical point there, where it reaches its minimum. Hence, for $0<\varepsilon<1$, the function
 $$y \mapsto y^2/2 + \varepsilon b(y/\varepsilon)$$ 
is Morse on $\R$ with three critical points $-c_\varepsilon < 0 < c_\varepsilon$ of respective indices $0$, $1$, $0$. Let $\rho : \R^+ \to [0,1]$ be a smooth function with support in $[0,1)$ such that $\rho=1$ near $0$. Define $\phi_\varepsilon : \D^2\times\D^{m-3} \to \R$~by
$$\phi_\varepsilon(x,y,z) = q(x,y) + \varepsilon \rho(x^2+y^2+\vert z\vert^2) \, b(y/\varepsilon) + Q(z).$$
It enjoys the following properties~:
\begin{itemize}
\item $\phi_\varepsilon(-x,-y,z)=\phi_\varepsilon(x,y,z)$.
\item For $\varepsilon>0$ sufficiently small, the function $\phi_\varepsilon$ has exactly three critical points~: $0$ of index $i+1$, and $(0, \pm c_\varepsilon,0,\dots,0)$ of index~$i$.
\item $\phi_\varepsilon(x,y,z)=q(x,y)+Q(z)$ on a neighborhood of $\partial(\D^2\times\D^{m-3})$.
\end{itemize}

Fix such an $\varepsilon > 0$ and set
$$g(t,x,y,z)=f(t,0,0,0)+\phi_\varepsilon(R^{t\pi}(x,y),z).$$
The function $g$ is leafwise Morse on $T$ and coincides with $f$ near $\partial T$. It has two leafwise critical circles in $T$ of respective indices $i$ (covering twice the original $C$) and $i+1$, whose stable and unstable manifolds are orientable.
\end{proof}

\section{Making foliated cobordisms conformally symplectic}\label{cobordism_sec}

Here, we deduce from the tools developed in the previous section and from the symplectization theorem for cobordisms (\cite{eliashberg_murphy_20}), a foliated version of the latter.

Let $W$ be a compact
 manifold of dimension $2n+1\ge 5$, whose smooth boundary
 $\partial W$ is splitted into two disjoint nonempty compact
  subsets $\partial_\pm W$. Let $\F$ be on $W$
   a cooriented codimension-$1$ foliation, transverse to
 $\partial W$. One has the induced foliations  $\partial\F=\F\vert_{\partial W}$
and  $\partial_\pm\F=\F\vert_{\partial_\pm W}$. For every leaf $L$ of $\F$,
put $\partial_\pm L=L\cap\partial_\pm W$.

Recall Definition \ref{ac_dfn}.
 A leafwise $2$-form $\omega\in\Omega^2(\F)$ is of course \emph{nondegenerate}
 if $\omega^n$ does not vanish.
 Such a form defines a leafwise orientation on $\F$, hence also on $W$ and on $\partial W$.
  
 Let  $\eta$ be a leafwise $1$-form on $\F$ which
  is {closed} ($d_\F\eta=0$). Then, every $\theta\in\Omega^*(\F)$
 has a \emph{leafwise Lichnerowicz differential} with respect to $\eta$
  $$d_\eta\theta=d_\F\theta-\eta\wedge\theta$$ 
  Just as in the nonfoliated case,
  $d_\eta^2=0$, hence the differential operator $d_\eta$ on $\Omega^*(\F)$
  defines some Novikov leafwise cohomology groups. We are interested in the relative ones: precisely,  $H_\eta^*(\F,\partial\F)$ is the cohomology of 
  $\Omega^*(\F) \times \Omega^{*-1}(\partial \F)$ under the differential operator
   $$(\theta,\theta')\mapsto(d_\eta\theta,\theta\vert_{\partial\F}-d_\eta\theta')$$

 \begin{TB} Let $W$, $\partial_\pm W$, $\FF$, 
 $\omega$, $\eta$ be as above. Let $a\in H_\eta^2(\F,\partial\F)$.
  Assume that $\F$ is taut (Definition \ref{taut_dfn}) and that
every leaf of $\F$ meets both $\partial_+W$ and $\partial_-W$.

 Then, there exist $\varpi\in\Omega^2(\F)$ and $\alpha\in\Omega^1(\partial\F)$
  such that
   \begin{itemize}
 \item $d_\eta\varpi=0$, and $\varpi\vert_{\partial\F}=d_\eta\alpha$;
 \item $(\varpi,\alpha)$ lies in the relative Novikov leafwise cohomology class $a$;
  \item  $\varpi$ is nondegenerate and
  homotopic
  to $\omega$ among the nondegenerate $2$-forms on $\F$;
 \item  $\alpha$ is a negative (resp. positive)
  overtwisted contact form on every leaf of $\partial_+\F$ (resp. $\partial_-\F$).
  \end{itemize}

 \end{TB}
  In particular, on every leaf $L$ of $\F$, the $2$-form
  $\varpi\vert_L$ is $\eta$-symplectic; and
 $\partial_+ L$ (resp. $\partial_-L$) is
  of concave (resp. convex)
  overtwisted contact type (Definition \ref{twisted_dfn_3}
  and Lemma \ref{hypersurface_lem})
  with respect to $\varpi\vert_L$. 
 \begin{rmk} Here, $\eta$ may be $d_\F$-exact, or even vanish identically.
 \end{rmk}
 
 \begin{proof}[Proof of Theorem $B$] The proof
  is essentially a ``foliated'' version of parts of the above
  proof of Theorem $A$, using the tools elaborated in
   Section \ref{morse_sec}. We begin with the case $a=0$. So, we are
   actually looking
   for a \emph{ leafwise $\eta$-Liouville form} on $\F$: a $\lambda\in\Omega^1(\F)$
   such that $d_\eta\lambda\in\Omega^2(\F)$ is nondegenerate.

  \medbreak
  Applying Propositions \ref{cancellation_pro},  \ref{orientation_pro} and \ref{ordered_pro},
one makes on $W$ a leafwise
Morse function $f$ (Definition \ref{lm_dfn}) such that
\begin{itemize}
\item  $\partial_-(W,f)=\partial_-W$
and $\partial_+(W,f)=\partial_+W$;
\item $f$ has no leafwise local extrema in $Int(W)$;
\item The  stable and unstable manifolds  of
$\crit(f,\F)$ are orientable.
\end{itemize}
Choose a leafwise pseudo-gradient $Z$ for $f$ on $W$ (Definition \ref{pdg_dfn}).

  \medbreak
\emph{Leafwise symplectization close to the critical locus ---}
  Thanks to the orientability of the (un)stable manifolds,
  every connected component $C$ of the critical locus of index $1\le i\le 2n-1$
   admits in $W$ a small compact 
  neighborhood $H_C$ which is a topological solid torus
  with (convex) cornered boundary, as follows. $H_C$ is diffeomorphic with
  $\S^1\times H_i$ where
  $H_i$ is $\D^i\times\D^{2n-i}$ minus a small open
  tubular neighborhood of the corner $\S^{i-1}\times\S^{2n-i-1}$;
  and the boundary splits as
  $$\partial H_i=\partial_+H_i\cup\partial_0H_i\cup\partial_-H_i$$
 where (writing $\partial_+H_C$ for $\S^1\times\partial_+H_i$
 and  $\partial_0H_C$ for $\S^1\times\partial_0H_i$ and
  $\partial_-H_C$ for $\S^1\times\partial_-H_i$):
  \begin{itemize}
  \item $\F\vert_{H_C}$ is the slice foliation parallel to the factor $H_i$;
 \item
 $f$ is leafwise constant on $\partial_+H_C$ and on $\partial_-H_C$;
 \item $Z$ enters (resp. exits) $H_C$  transversely through $\partial_+H_C$
 (resp. $\partial_-H_C$);
 \item  $Z$ is tangential to the $\I$ factor on
  $$\partial_0H_C\cong\S^1\times\S^{i-1}
 \times\S^{2n-i-1}\times\I$$
 \end{itemize}

 Since the leaves of $\F\vert_{H_C}$ are topological disks, $\eta$
 admits a $d_\F$-primitive $u$ on $H_C$. Extending $u$ to a smooth function over $W$,
 and  changing
  $\eta$ to $\eta-d_\F u$ on $W$,  we can without loss of generality
  arrange that $\eta=0$ on $\nb_W(H_C)$ (here we use an obvious foliated version of
  Remark
  \ref{cc_rmk}).
 Also, $U(n)/SO(2n)$ being simply connected,
 after a homotopy of $\omega$, we can arrange that $\omega\vert(s\times H_i)$
 does not depend on $s\in\S^1$.
  Hence, in $H_C$, we actually look for a \emph{genuine} Liouville form on a single slice $H_i$. 
 Such a form is given by
  the symplectization of cobordisms
  \cite{eliashberg_murphy_20}. One gets
   a leafwise Liouville form $\lambda_C$ on $\nb_W(H_C)$, positive with respect
   to the orientation of the leaves by $\omega^n$; and whose
 dual Liouville vector field is positively colinear to
  $Z$ on $\nb_{W}(\partial H_C)$; moreover,
 $\lambda_C$ induces an overtwisted
 contact structure on every leaf of $\F\vert_{\partial_\pm H_C}$.
  \medbreak
  
\emph{Construction of a leafwise even contact structure away from the critical locus ---}
For short, put $H=\cup_CH_C$ and write $\lambda_H$ for
the leafwise $1$-form equal to $\lambda_C$ on each $H_C$.

 Consider the codimension-$2$ foliation $\LL$
of $W'=W\setminus Int(H)$ by the level hypersurfaces of $f$ in the leaves of $\F$,
cooriented by $df$.
Rescale $Z$ such that $Z\cdot f=-1$ on $W'$.

  The pair $(\iota_Z\omega,\omega)$ restricts, on $\LL$, to
  a leafwise almost contact structure (Definition \ref{ac_dfn}).
  The h-principle for overtwisted contact structures
  on foliations
(\cite{borman_eliashberg_murphy_15}, Theorem 1.5) provides for the foliation $\LL$
 a leafwise contact, cooriented $(2n-2)$-plane field $\xi\subset T\LL$ such that
 \begin{itemize}
 \item $\xi$ lies in the leafwise almost
contact class of  $((\iota_Z\omega)\vert_\LL,\omega\vert\LL)$;
 \item $\xi$ is an overtwisted contact structure in every leaf of
$\LL$;
\item $\xi$ coincides with $\ker(\lambda_H)\cap T\LL$ near $H$.
\end{itemize}

Just like in the proof of Lemma \ref{even_contact_lem},
after the Gray stability theorem, there is a unique vector field $X$ on  $W'$,
tangential
to $\xi$, such that $Z'=Z+X$ preserves $\xi$. In particular, $X$ vanishes on some
neighborhood of $\partial H$.
Change $Z$ to $Z'$ on $W'$ and put $\varepsilon=\R Z+\xi$.
So,  in every leaf $F$ of $\F\vert_{W'}$, the hyperplane field
 $\varepsilon$ is an even contact structure (Section \ref{elements_sec}),
 represented by
 $\lambda_H\vert_F$ close to $\partial H\cap F$,
 and whose characteristic foliation is positively spanned by $Z$.

  \medbreak

   \medbreak
   \emph{Construction of a leafwise $\eta$-Liouville form ---}
By means of a partition of the unity, make  over $W$ a leafwise 1-form $\lambda\in
\Omega^1(\F)$ such that
\begin{itemize}
\item $\lambda$ represents $\varepsilon$ in each leaf of $\F\vert_{W'}$;
\item $\lambda$ coincides with $\lambda_H$ on some open neighborhood $V$ of $H$
in $W$.
 \end{itemize}

   By an easy modification of $f$ in a small neighborhood of
  $H$,
 one gets a smooth real function $g$ on $W$,
   constant on every leaf of $\F\vert H$, and such that
  $Z\cdot g<0$ on $W\setminus H$.
 After multiplying $g$ by a large enough positive constant, one arranges moreover
  that  on $W\setminus V$:
  \begin{equation}\label{g_eqn}
  Z\cdot g<\chi_\lambda-\eta(Z)
  \end{equation}

  After Lemma \ref{derivate_lem}, $\chi(\lambda)>\eta(Z)$ on $V\setminus H$.
   After Equation (\ref{derivate_eqn}) and Inequation (\ref{g_eqn}), 
  changing $\lambda$ to $e^{-g}\lambda$, one can moreover
   arrange that $\chi(\lambda)>\eta(Z)$
   on the all of
    $W\setminus H$. Hence, $\lambda$ is leafwise $\eta$-Liouville on $W\setminus H$;
    and also on $H$, being there a {leafwise locally constant} multiple
 of $\lambda_H$.
  By construction, $\varpi=d_\eta\lambda$ satisfies all the properties of Theorem $B$
 in the exact case $a=0$.

 \medbreak
   \emph{General case --- }
In order 
to obtain a leafwise $\eta$-symplectic form in a given relative cohomology class $a$,
 we proceed as in the non-foliated case: $a$ is represented by a pair $(\omega',0)$
 such that $\omega'\in\Omega^2(\F)$ is $d_\eta$-closed.
  For a large enough positive real constant $K$, the leafwise forms
   $$\varpi=\omega' + K d_\eta\lambda$$ $$\alpha=K\lambda\vert_{\partial\F}$$
    satisfy the required properties.
\end{proof}

\section{Deforming foliations into contact structures}\label{deformation_sec}

 Consider the problem of approximating a foliation
    by contact structures, which was solved in the $3$-dimensional case by Eliashberg and Thurston in their seminal monography \cite{eliashberg_thurston_98}.
  
On a compact oriented manifold $M$
 of dimension $2n+1\ge 5$, let $\F$ be
 a cooriented codimension-$1$ foliation transverse to the boundary.
 
 The simplest way to such an approximation is a so-called \emph{linear deformation}:
  that is, the foliation $\F$, being cooriented, is defined by a global non-vanishing $1$-form $\alpha\in\Omega^1(M)$; and one looks for a $1$-form $\lambda\in\Omega^1(M)$
 such that $\alpha_t=\alpha+t\lambda$ is contact for every small enough positive $t$.

  The actual geometric nature of the problem will appear through an
   elementary computation.
  
  Recall \cite{candel_00} that the integrability of $\alpha$
  amounts to the existence of a $1$-form $\eta$ on $M$
   such that $d\alpha=\eta\wedge\alpha$;
 that $\eta$ is then leafwise closed;
  the integral of $\eta$ on every tangential loop being the logarithm of
  the linear holonomy of the loop. 
  The restriction $\eta\vert_\F$ is uniquely determined by $\alpha$.
  One may call $\eta$ a \emph{holonomy form} associated to $\alpha$.
   Then, for any smooth function $F$ on $M$,
 $\eta + dF$ is a holonomy form
 associated with $e^F\alpha$.
  The cohomology class of $\eta\vert_\F$
  in $H^1(\F)$ (recall Definition \ref{ac_dfn}) thus depends only on the foliation $\F$, not on the choice of $\alpha$. 
  
  \begin{lem}\label{deformation_lem} Let $\F$, $\alpha$, $\eta$ be as above.
  If $\lambda\in\Omega^1(M)$ and if $\lambda\vert_\F$ is leafwise
  $\eta$-Liouville,
   then $\alpha+t\lambda$ is contact for every small enough positive $t$.
\end{lem}

\begin{proof}[Proof of Lemma \ref{deformation_lem}]
  For $\theta\in\Omega^*(M)$, we use the notation $d_\eta\theta$ for $d\theta-\eta\wedge\theta$,
although $\eta$ being in general {not} globally closed on $M$,
the operator $d_\eta^2$ is not in general a differential operator on $\Omega^*(M)$:
its square vanishes \emph{in restriction to $\F$.}
   One gets straightforwardly
  
  $$(d_{\eta}\lambda)^n=(d\lambda)^n-
n\eta\wedge\lambda\wedge(d\lambda)^{n-1}$$
\begin{equation}\label{deformation_eqn}
\alpha_t\wedge(d\alpha_t)^n=t^n\alpha\wedge(d_{\eta}\lambda)^n
+t^{n+1}\lambda\wedge(d\lambda)^n
\end{equation}
Hence, a \emph{sufficient} condition for $\alpha_t$ to be
 contact for every small enough positive $t$
 is that $\alpha\wedge(d_{\eta}\lambda)^n$ be a volume form on $M$.
\end{proof}

  \begin{rmk}\label{invariance_rmk}
   The existence of such a form $\lambda$,
      and the leafwise conformal class of $d_\eta\lambda\vert_\F$,
      depend only on the cooriented foliation $\F$, not
    on the choice of $\alpha$.
   Indeed, let $\lambda$ be leafwise $\eta$-Liouville.
  Change $\alpha$ to $e^F\alpha$ for some smooth function $F$
   on $M$; then,
   $e^F\lambda$ is leafwise $(\eta+dF)$-Liouville (Remark \ref{cc_rmk}).
  
     \end{rmk}

\begin{dfn}\label{rich_dfn}
 We call $\F$
  \emph{holonomous} if every minimal set contains a
tangential loop whose linear holonomy is nontrivial. 
\end{dfn}

In
dimension $3$,
Eliashberg and Thurston proved that
 the condition of being holonomy-rich is necessary and sufficient
 for a cooriented foliation to admit
 a linear deformation (\cite{eliashberg_thurston_98}, Theorem 2.1.2.).  In the higher dimensions,
 being holonomy-rich
 remains \emph{necessary}  for the existence of a leafwise holonomy-Liouville form:
 indeed, in this direction,
  the arguments of \cite{eliashberg_thurston_98} hold as well in all dimensions
  (recently, Lauran Toussaint has given an alternative proof \cite{toussaint_20}).

\begin{TC}
On a compact manifold
$M$ of dimension $2n+1\ge 5$ with smooth boundary, 
let $\F$ be a cooriented codimension-one foliation which is transverse to $\partial M$,
 taut (Definition \ref{taut_dfn}), holonomous,
and such that every leaf meets $\partial M$.
Assume that $\F$ admits a nondegenerate leafwise $2$-form $\omega$.
Choose on $M$ a nonsingular $1$-form $\alpha$ defining $\F$, and a $1$-form  $\eta$
such that $d\alpha=\eta\wedge\alpha$.

 Then, $\F$ admits a leafwise $1$-form $\lambda\in\Omega^1(\F)$ such that
  
  \begin{enumerate}
  \item
  $d_\eta\lambda$ is nondegenerate,
 and homotopic to $\omega$ as a nondegenerate leafwise $2$-form;
\item
$\lambda$ restricts to a negative overtwisted contact form on
every leaf of $\partial\F=\F\vert_{\partial M}$.
  \end{enumerate}

\end{TC}
In consequence (Lemma \ref{deformation_lem}),
 $\F$ admits a linear deformation into a contact structure
 for which the leaves of $\partial\F$
(cooriented by
a vector tangential to $\F$ and
 pointing outward $M$ followed by a vector tangential to $\partial M$
and positively transverse to $\F$)
are negative contact submanifolds.

\begin{rmk}
 The tautness hypothesis in Theorem $C$ can certainly be weakened, and maybe
suppressed.
 This hypothesis would be a serious restriction in dimension $3$,
since after the classical
 Novikov theorem, many $3$-manifolds do not admit any taut foliation.
But it is known that
 the case is different on a manifold $M$ of higher dimension:
  every cooriented hyperplane field on $M$ is homotopic to a smooth
 foliation, with nontrivial linear holonomy, and whose leaves are dense in $M$
\cite{meigniez_17}. (One may even prescribe $\F\vert\partial M$.)
 Such a foliation is in particular taut and holonomous.
Since, moreover, both properties of
 holonomousness and tautness are clearly open in the space
of codimension-one foliations on $M$, we conclude that for
every homotopy class of almost symplectic cooriented
hyperplane fields on $M$, our Theorem $C$ applies to a \emph{nonempty open subset} of
the space of foliations lying in this class.
\end{rmk}

 \begin{proof}[Proof of Theorem $C$]
 The foliation being holonomous means that
  $M$ contains in its interior a finite disjoint union $\Gamma$
of embedded oriented circles (one in each minimal set) such that
\begin{itemize}
\item Each component $\gamma$ of $\Gamma$ is tangential to $\F$,
 and $h_\gamma=\int_{\gamma}\eta\neq 0$;
\item The closure of every leaf of $\F$
contains at least one component of $\Gamma$.
\end{itemize}

After changing each loop $\gamma$ to a nearby small perturbation, in the same leaf, of
a positive or negative multiple of $\gamma$,
 one can moreover arrange that $h_\gamma<-c_n$
(the constant of Lemma \ref{toric_lem}). In particular, $h_\gamma<0$ (Remark \ref{positive_note}).
\medbreak
We now define $\lambda$ close to $\gamma$.

The
linear holonomy being nontrivial, $\gamma$ admits a small compact tubular
 neighborhood
$T_\gamma
\cong\S^1\times\D^{2n}$
in which $\F$ coincides with a standard, somehow linear model: namely, $\F$ is
 defined in this solid torus by the nonvanishing $1$-form
\begin{equation}\label{model_eqn}
\alpha= dx_{2n}-h x_{2n}\theta
\end{equation}
where $h=h_\gamma$, where $\theta$ denotes the positive unit volume form on $\S^1$,
 and where $x_1$, \dots, $x_{2n}$ are the standard coordinates on the compact unit ball
 $\D^{2n}\subset\R^{2n}$.
Note that
\begin{itemize}
\item $\eta=h\theta$ is a holonomy form associated with $\alpha$ in $T_\gamma$;
\item $\F\vert_{T_\gamma}$
 has a unique compact leaf $L_\gamma\cong\S^1\times\D^{2n-1}$,
defined by $x_{2n}=0$;
 \item
 $\F$ is transverse to $\partial T_\gamma\cong\S^1\times\S^{2n-1}$;
 \item The induced foliation
$\F\vert_{\partial T_\gamma}$ consists
 of two $(2n)$-dimensional Reeb components with common boundary $\partial
 L_\gamma$,
the one on $x_{2n}\ge 0$, the other on $x_{2n}\le 0$.
\end{itemize}

Consider in $T_\gamma$
 the projection $\pi:T_\gamma\to L_\gamma$ parallelly to the $x_{2n}$-axis,
 and the $1$-form
$$\rho=
(1-x_1^2-\dots-x_{2n-1}^2)\eta+x_1dx_1+\dots+x_{2n-1}dx_{2n-1}$$

The restriction $\rho\vert L_\gamma$ draws on
 $L_\gamma$ a  ($2n$)-dimensional Reeb component; while
 in restriction to $\partial T_\gamma$ one has
  $$\rho=\pi^*(\rho\vert L_\gamma)=-x_{2n}\alpha$$
 
Endow the solid torus $L_\gamma\cong\S^1\times\D^{2n-1}$, oriented by $\omega$,
 with the $1$-form $\lambda$ and with the vector field $Z$ given by
Lemma \ref{toric_lem} which \emph{enters} $L_\gamma$
 transversely through $\partial L_\gamma$.
Since $\theta(Z)=1$, after pushing $\lambda$ and $Z$ by a self-diffeomorphism
of $\S^1\times\D^{2n-1}$ preserving the projection to $\S^1$, we can arrange that
moreover, $\rho(Z)<0$
on $L_\gamma$.
Exdent $\lambda$ over $T_\gamma$ as the leafwise $1$-form $\pi^*(\lambda)\vert_{\F}$
(also denoted by $\lambda$); extend $Z$ over $T_\gamma$ as the vector field
that lifts $Z$ through $\pi$
tangentially to $\F$ (also denoted by $Z$). After 
Lemma \ref{toric_lem}, $\lambda$ is leafwise
 $\eta$-Liouville
in $T_\gamma$, and its  $\eta$-dual vector field is positively colinear to $Z$.
The space $U(n)/SO(2n)$ being simply connected, $d_\eta\lambda$ is homotopic
to $\omega\vert_{T_\gamma}$ as a nondegenerate leafwise $2$-form on
 $\F\vert_{T_\gamma}$.

The function $\rho(Z)$ being negative on $L_\gamma$, the vector field
 $Z$ enters transversely $T_\gamma$ through $\partial T_\gamma$.
After Lemma \ref{toric_lem}, the contact form $\lambda\vert_{\partial L_\gamma}$
is overtwisted. For every leaf $L$ of $\F\vert_{T_\gamma}$, since $\partial L$
 accumulates
on $\partial L_\gamma$ in $\partial T_\gamma$, 
it follows that $\lambda\vert_{\partial L}$
is overtwisted as well (any overtwisted ball in $\partial L_\gamma$ can be pushed
into $\partial L$ by an isocontact embedding close to the identity, after Gray's stability theorem). In other words, $\partial L$
is of overtwisted contact type and concave with respect to $d_\eta\lambda\vert_L$
(Definition \ref{twisted_dfn_3}, Lemma \ref{hypersurface_lem}).
\medbreak
After $\lambda$ has thus been constructed over the union $T$
of the $T_\gamma$'s, Theorem $B$ then allows one to complete
 the construction over $M$.
  Here are some precisions.
  
  In the cobordism $W=M\setminus Int(T)$ 
between $\partial_-W=\partial T$
and $\partial_+W=\partial M$, the foliation $\F\vert_W$ is transverse to $\partial W$,
taut,
and every leaf meets $\partial_-W$ and $\partial_+W$.
Extend $\alpha$
and $\eta$ from $T$ to $M$, such that $\alpha$ defines $\F$ over $M$, and that $\eta$
is a holonomy form associated to $\alpha$ over $M$.

After Theorem $B$, there is 
a $\eta$-Liouville leafwise $1$-form $\lambda'\in\Omega^1(\F\vert_W)$
 restricting, on every leaf $\ell$ of $\partial_+\F=\F\vert_{\partial M}$
  (resp. $\partial_-\F=\F\vert_{\partial T}$), to
 an 
overtwisted contact form which is negative (resp. positive) ---
here, $\ell$ is cooriented as a component of the boundary of a leaf of $\F\vert_W$ ---
and such that
 moreover, $d_\eta\lambda'$ is homotopic to $\omega$ as a nondegenerate leafwise $2$-form
on $\F\vert_W$.

There remains to paste the two pieces.
 After the h-principle for overtwisted leafwise contact structures on foliations
\cite{borman_eliashberg_murphy_15}, there is an isotopy $\phi$
 of $\partial T$ tangential
to $\F\vert_{\partial T}$ and such that  $\lambda$
and $\phi^*(\lambda')$ define the same leafwise contact structure on $\F\vert_{\partial T}$.

  On $\nb_T(\partial T)$ (resp. $\nb_W(\partial T)$), the leafwise $1$-form
  $\lambda$ (resp. $\lambda'$) defines for $\F$ a 
   leafwise even contact structure $\varepsilon$
 (resp. $\varepsilon'$) whose
  characteristic foliation $\ZZ$ (resp. $\ZZ'$)
  is a $1$-dimensional foliation
  transverse to $\partial T$.
 Extend $\phi$ to an isotopy of $W$ tangential to $\F\vert_W$, still denoted by $\phi$,
 such that $\ZZ$ and $\phi^*(\ZZ')$ match along $\partial T$,
 giving a global smooth $1$-dimensional foliation on $\nb_M(\partial T)$.
 Thus, $\varepsilon$ and $\phi^*(\varepsilon')$ give
  a global  even contact structure on $\nb_M(\partial T)$.
 Then, since $\phi^*(\eta\vert_\F)$ is cohomologous to $\eta\vert_\F$
  in $H^1(\F\vert_W)$
 (recall Definition \ref{ac_dfn}),
  multiplying $\phi^*(\lambda')$ by a convenient positive
function, one gets again a $\eta$-Liouville leafwise $1$-form $\lambda''$ on $W$, see
 Remark \ref{invariance_rmk}.
 Finally, in view of Lemma \ref{hypersurface_lem},
after multiplying again
 $\lambda''$ by a convenient positive function, the $\eta$-Liouville leafwise $1$-forms
  $\lambda$ and $\lambda''$ moreover
match along $\partial T$, and define a global $\eta$-Liouville leafwise $1$-form for $\F$
over $M$.
\end{proof}

\begin{rmk}\label{limitation_rmk}
Our method does not seem
to be able to produce contactizing linear deformations for taut foliations on closed manifolds. Precisely, Proposition
\ref{tightness_pro} will show that
starting from Theorem $B$,  the concave boundary cannot be eliminated in the same way as we have eliminated the convex boundary and got Theorem $C$.

 This problem raises the general
  question of whether the Eliashberg-Gromov tightness criterion for fillable contact structures admits the following
\emph{foliated and conformal} analogue for leafwise contact structures in all dimensions. 
 \end{rmk}
 
\begin{qst}\label{tightness_qst} Let $M$ be an oriented compact manifold of dimension $2n+1\ge 5$ with nonempty smooth boundary, endowed with a cooriented codimension-$1$ foliation $\F$, transverse to $\partial M$. Let $\alpha$, $\eta$, $\lambda$ be, respectively, a defining form for $\F$, an associated holonomy $1$-form, and a $\eta$-Liouville 
leafwise $1$-form which restricts to a positive contact form on every leaf of $\partial\F$.
 Does it follow that $\lambda$ is tight on every leaf of $\partial\F$~?
\end{qst}

\begin{pro}\label{tightness_pro} The answer to Question \ref{tightness_qst} is positive for the model foliation $\F$ defined by Equation
 (\ref{model_eqn}) on $T^{2n+1}=\S^1\times\D^{2n}$,
 for every $n\ge 2$ and $h \neq 0$.
\end{pro}

\begin{proof}[Proof of Proposition \ref{tightness_pro}] Let us begin, to fix ideas,
with the case $n=2$.
In this case, Proposition \ref{tightness_pro}
 is a simple application of the Eliashberg-Gromov tightness criterion (\cite{eliashberg_91}, \cite{eliashberg_thurston_98} and \cite{gromov_85}) in the noncompact framework:
  every contact $3$-manifold $(\partial M, \xi)$ which bounds a symplectic 
 $4$-manifold $(M, \om)$ with bounded symplectic geometry at infinity is tight.

\medskip
Consider, on $\partial \F$, the leafwise contact structure 
$$\xi = \ker(\lambda|_{\partial T^5}).$$
Let $L_0$ and $\ell_0$ denote the compact leaf of $\F$
 and its boundary, respectively. Since every other leaf $\ell$ of
 $\partial\F$ accumulates on $\ell_0$, if $\xi \vert_{\ell_0}$ were overtwisted,
then $\xi\vert_\ell$ would also be overtwisted.
 Hence, it suffices to prove that $\xi|_\ell$ is tight for every 
 \emph{non-compact} leaf $\ell$ of $\partial \F$. 

\medskip
Every noncompact leaf $L$ of $\F$ being without holonomy,
 the holonomy $1$-form $\eta=h\theta$ is exact on $L$,
 hence (Remark \ref{cc_rmk})
  $\lambda\vert_L$ is conformal to a genuinely Liouville form on $L$.
   In fact, the function $\ln(\vert x_4\vert)$ is a primitive of $\eta$ in
   restriction to every noncompact leaf.
    The fact that this function is bounded from above on every such leaf, which is specific to foliations of very simple dynamics like $\F$, seems to be crucial in the proof.

\medskip
We now show that any non-compact leaf $L$ has bounded symplectic geometry. The $2$-form $d_\eta\lambda$ on $T^5$, being leafwise nondegenerate, admits a leafwise almost complex structure $J$ (an automorphism of the vector bundle
$T\F$ such that $J^2=-\id$) which preserves $\xi$ at every point of $\partial T^5$, and such that 
$$g = d_\eta\lambda(\cdot, J\cdot)$$
defines a Riemannian metric on every leaf $L$ of $\F$. By compacity of $T^5$, the metric $g|_L$ has \emph{bounded geometry}, meaning that $g\vert_L$ is a complete Riemannian metric on $L$, whose injectivity radius is bounded away from zero,
 and whose sectional curvature is bounded.

\medskip
The $1$-form $\lambda' = \vert x_4\vert\mun\lambda|_L$ is genuinely Liouville on $L$;
 and the exact symplectic form $d\lambda'=\vert x_4\vert\mun d_\eta\lambda$ \emph{dominates} $\xi\vert_\ell$, in the sense that $d\lambda'$ is nondegenerate on $\xi|_\ell$ at every point of $\ell$. Set 
$$g' = \frac{g}{\vert x_4\vert} = d\lambda'(\cdot, J\cdot).$$
The metric $g'$ is compatible with $d \lambda'$. It remains
 to verify that $g'$ is of bounded geometry on $L$.

 One has $g' \geq g$, so that $g'$ is complete as well on $L$.
It is convenient to introduce the solid cylinder
$$C = \left[-1, 1\right] \times \D^4$$
and, for every $s \in \S^1 = \R/\Z$, the immersion
$$j_s : C \hookrightarrow T^5 : (t,x)\mapsto(s+t,x)$$
The foliation $\CC=j_s^*(\F)$ of $C$ does obviously not depend on $s$.
Consider on $\CC$ the smooth family of leafwise Riemannian metrics, pa\-ra\-me\-tri\-zed by $s \in \S^1$
$$g_s = e^{-ht}j_s^*(g)$$
Let $\iota_0>0$ be the minimum of their injectivity radii, and $\sigma_0<+\infty$ be the maximum of the absolute values of their sectional curvatures.
    
Consider the leaf $L_a$ of $\CC$ through the point $(0,a)$, with $a = (a_1, ..., a_4)$ in $\D^4$ and $a_4\neq 0$. Since $x_4|_{L_a} = a_4 e^{ht}$, on $L_a$, one has
$j_s^*(g') = {\vert a_4 \vert}\mun g_s$; and
the following bounds follow at once on $L_a$~:
$$\iota(j_s^*(g'))=\vert a_4\vert\mun\iota(g_s)\ge\iota_0$$
$$\vert\sigma(j_s^*(g'))\vert=\vert a_4\vert\vert\sigma(g_s)\vert\le\sigma_0.$$

The proof of Proposition \ref{tightness_pro} in the higher dimensions is much alike,
 but instead of the original Eliashberg-Gromov tightness criterion,
one applies Niederkr\"uger's tightness criterion \cite{niederkruger_06} (see also \cite{borman_eliashberg_murphy_15} paragraph 10),
  in the noncompact framework, under the hypothesis of bounded geometry
  (this noncompact version of the criterion does not appear in the litterature, but
  the generalization is
  straightforward).

\end{proof}

\ \\

M\'elanie Bertelson: {\tt Melanie.Bertelson@ulb.be} --- {D\'epartement de Math\'e\-ma\-tique,
Universit\'e libre de Bruxelles --- Boulevard du Triomphe --- 1050 Bruxelles --- Belgique.} \\

Ga\"el Meigniez: {\tt Gael.Meigniez@univ-amu.fr} --- {Aix-Marseille Universit\'e, I2M (CNRS UMR 7373, Centrale Marseille) ---
13453 Marseille --- France.}

\end{document}